\documentclass[a4paper, 12pt]{article}
\pdfoutput=1
\usepackage{hyperref}
\usepackage{amsfonts}
\usepackage{amsmath}
\usepackage{amsthm}
\usepackage{algorithm}  
\usepackage{algpseudocode}  
\usepackage{amsmath}  
\usepackage{graphicx}
\usepackage{subfigure}
\usepackage{booktabs}  
\usepackage{multicol}  
\usepackage{multirow}
\usepackage{threeparttable}  
\usepackage{makecell}
\usepackage{caption2}

\usepackage{lipsum}

\date{} 

\newtheorem{theorem}{Theorem}

\newtheorem{definition}{Definition}

\title{Tensor Completion via Tensor QR Decomposition and $L_{2,1}$-Norm Minimization\footnote{Research supported by the National Natural Science Foundation of China under Grant (11801418).}}

\begin{document}
\maketitle

\begin{center}

{Yongming Zheng \footnote{Email: zym609629@foxmail.com (Yongming Zheng)}, An-Bao Xu \footnote{Corresponding Author: xuanbao@wzu.edu.cn (An-Bao Xu)}} \\[2mm]
\emph{College of Mathematics and Physics, Wenzhou University, Zhejiang 325035, China.}

\end{center}

\begin{abstract}
In this paper, we consider the tensor completion problem, which has many researchers in the machine learning particularly concerned. Our fast and precise method is built on extending the $L_{2,1}$-norm minimization and Qatar Riyal decomposition (LNM-QR) method for matrix completions to tensor completions, and is different from the popular tensor completion methods using the tensor singular value decomposition (t-SVD). In terms of shortening the computing time, t-SVD is replaced with the method computing an approximate t-SVD based on Qatar Riyal decomposition (CTSVD-QR), which can be used to compute the largest $r \left(r>0 \right)$ singular values (tubes) and their associated singular vectors (of tubes) iteratively. We, in addition, use the tensor $L_{2,1}$-norm instead of the tensor nuclear norm to minimize our model on account of it is easy to optimize. Then in terms of improving accuracy, ADMM, a gradient-search-based method, plays a crucial part in our method. Numerical experimental results show that our method is faster than those state-of-the-art algorithms and have excellent accuracy.\\
\end{abstract}


\section{\large\bf Introduction}
A multidimensional array, also named tensor, generalizes vectors and matrices \cite{KB2009} \cite{H2012}. A large number of  researchers attach great importance to tensor completions \cite{LMWY2013} \cite{GRY2011} \cite{KS2008} \cite{LLR1995} \cite{X2020}, because it arises in all kinds of applications, such as in dimensionality reduction \cite{LLR1995} \cite{ZJ2016}, in computer vision \cite{LMWY2013} \cite{X2020} \cite{LTYL2016}, in signal processing \cite{ZEAHK2014} and data mining \cite{KS2008} \cite{SPLCLQ2009}.

Similar to the problem of matrix completion \cite{XWX2020} \cite{XX2017}, conventional method formulates the problem of tensor completion as follows: 
\begin{equation}\label{1.1}
\min_{\mathcal{X}\in \mathbb{R}^{n_{1}\times n_{2}\times n_{3} }   } \text{rank}_{\text{t}} \left ( \mathcal{X}  \right )   \quad \text{s.t.} \quad \mathcal{X}_{i,j,k}= \mathcal{M}_{i,j,k}, \quad \left ( i,j,k \right ) \in  \Omega   
\end{equation}
where $ \mathcal{X} \in \mathbb{R}^{n_{1}\times n_{2} \times n_{3}}$ is a low-tubal-rank tensor,  $ \text{rank}_{ \text{t}} \left ( \mathcal{X}  \right )$ is the tubal-rank of $ \mathcal{X}$ which is determined by t-SVD \cite{ZEAHK2014} \cite{KBHH2013} \cite{LFCLLY2019}, $ \mathcal{M} \in \mathbb{R}^{n_{1} \times n_{2} \times n_{3}}$ is an incomplete tensor, and $ \Omega$ is the set of observed locations. However, as we all know that Problem (\ref{1.1}) is N-P hard, because the tubal-rank function ${f} \left ( \mathcal{X}  \right ) = \text{rank}_{\text{t}} \left ( \mathcal{X}  \right )$ is not convex so that it is difficult to optimize. Fortunately, there is a convex surrogate function \cite{LFCLLY2019} for tubal-rank function, and Problem (\ref{1.1}) can be modified as follows:
\begin{equation}\label{1.2}
\min_{\mathcal{X}\in \mathbb{R}^{n_{1}\times n_{2}\times n_{3} }   } \left \| \mathcal{X}  \right \| _{\ast }   \quad \text{s.t.} \quad \mathcal{X}_{i,j,k}= \mathcal{M}_{i,j,k},\quad \left ( i,j,k \right )\in  \Omega   
\end{equation}
where $\left \| \mathcal{X}  \right \| _{\ast }$ is the tensor nuclear norm \cite{LFCLLY2019}. Still, if we want to optimize this model, we will need to decompose $\mathcal{X}$ at iteration step. Compared with CANDECOM/PARAFAC decomposition (CP) \cite{KB2009} \cite{SK2008} and Tucker decomposition \cite{LMWY2013}, tensor singular value thresholding \cite{LFCLLY2019} (t-SVT) based on t-SVD is one of the greatest and the most widely used methods. However, singular value decomposition (SVD) takes too many computations so that t-SVT can not be fast. In order to reduce the amount of computation and increase tensor completional speed, we have to find an efficient method which has less computational cost to replace t-SVD. 

Fortunately, inspired by \cite{LDYCJH2019} \cite{KD2011}, especially \cite{KD2011}, which recently proposed tensor product (t-product) whose details can be seen in Definition \ref{t-product}, we propose a method to decompose a thrid-order tensor, which is just like t-SVD via QR decomposition, named CTSVD-QR. Because CTSVD-QR and t-SVD are essentially based on QR decomposition and SVD respectively, besides, QR decomposition is much faster than SVD \cite{HMT2011}. Thus, our CTSVD-QR method is faster than t-SVD. At the same time, considering the extension of $L_{2,1}$-norm minimization \cite{NHCD2010} \cite{HLZ2017} of matrix, we propose tensor $L_{2,1}$-norm, defined in (\ref{L2,1-norm of tenosr}), to replace the tensor nuclear norm in Problem (\ref{1.2}). Tensor $L_{2,1}$-norm minimization problem is composed of several subproblems, whose solutions contain the t-SVT method, so that the optimizing efficiency and accuracy can be improved in this way.

Typically, investigating a fast and precise tensor completion method is a meaningful public conundrum. Happily, t-product provides a convenient bridge between tensor completion problem and matrix completion problem, which provides a valuable solution to this challenge.

In fact, our CTSVD-QR can also decompose a third-order tensor into the tensor product of three low-tubal-rank tensors analogous to t-SVD, and will faster than t-SVD. Additionally, under the framework of tensor tri-factorization, our tensor $L_{2,1}$-norm is easier than the nuclear norm to optimize. Thus, the purpose of this paper is to propose a fast and precise tensor completion method combining with tensor QR decomposition and tensor $L_{2,1}$-norm to work out the fore-mentioned public conundrum.

The main contributions of our paper are:
\begin{itemize}
\item A new method, which can decompose a third-order tensor iteratively into three low-tubal-rank tensors approximate t-SVD via QR decomposition, named CTSVD-QR, is proposed and is faster than t-SVD.
\item The tensor $L_{2,1}$-norm is proposed and is proven that the tensor nuclear norm is the lower bound on it. So it can be a substitute for tensor nuclear norm. Besides, we can also get the solution of tensor $L_{2,1}$-norm minimization problem.
\item Combining with tensor QR decomposition and tensor $L_{2,1}$-norm, a new tensor completion method, named TLNM-TQR, is proposed. The combination of these two methods makes optimization faster and easier. Numerical experimental results show that this method is much faster than some excellent algorithms in recent years and has an outstanding precision.
\end{itemize}

\section{\large\bf Notations and preliminaries}
\subsection{Notations}
In this paper, we denote vectors by lowercase boldface letters, $\mathbf{a}$, and matrices by uppercase boldface letters, $\mathbf{A}$, respectively. Tensors are represented in boldface Euler script letters. For instance, $\mathcal{A} \in \mathbb{C} ^{n_{1}\times n_{2}\times n_{3}   } $ is used to denote a third-order tensor, and its $\left ( i,j,k \right ) $th entry is represented as $\mathcal{A} _{i,j,k} $ or $a _{i,j,k} $. The $i$th horizontal, lateral and frontal slice (see definitions in \cite{KB2009}) of third-order tensor $\mathcal{A}$ is denoted by the Matlab notation $\mathcal{A}\left ( i,:,: \right ) $, $\mathcal{A}\left ( :,i,: \right ) $ and $\mathcal{A}\left ( :,:,i \right ) $ respectively. In general, $\mathcal{A}^{\left ( i \right ) } $ denotes the frontal slice $\mathcal{A}\left ( :,:,i \right )$ compactly, and $\mathcal{A}\left ( i,j,: \right ) $ denotes the tube of $i,j$ in the third-order tensor dimension.

For any $\mathcal{A}, \mathcal{B} \in \mathbb{C} ^{n_{1} \times n_{2} \times n_{3}   } $, their inner product is denoted as $\left \langle \mathcal{A},\mathcal{B}   \right \rangle =  {\textstyle \sum_{i=1}^{n_{3}}} \left \langle \mathcal{A}^{ \left ( i \right ) },\mathcal{B}^{\left ( i \right ) }   \right \rangle $, the complex conjugate of $ \mathcal{A}$ is denoted as conj$ \left ( \mathcal{A}  \right ) $, and the conjugate transpose of $ \mathcal{A}$ is denoted as $ \mathcal{A}^{\ast}$. Additionally, the Frobenius norm of tensor is denoted as $\left \| \mathcal{A}  \right \| _{F } = \sqrt{ {\textstyle \sum_{ijk}}\left | a_{ijk}^{2}  \right  |  }=\sqrt{\left \langle \mathcal{A},\mathcal{A}   \right \rangle} $. 


\subsection{Discrete Fourier Transformation}
The discrete fourier transformation (DFT) is the importance bridge between tensor product and matrix product. DFT gives a new tensor product that likes matrix product and makes it cost less calculation. We denote $\hat{\mathbf{v}} $ by 
\begin{center}
	$\hat{\mathbf{v}} = \mathbf{F} _{n} \mathbf{v}\in \mathbb{C} ^{n} $,
\end{center}
where $\mathbf{F} _{n}$ is the DFT matrix denoted as 

\begin{center}
	$\mathbf{F} _{n} = \begin{bmatrix}
	1& 1 & 1& \cdots &1\\
	1& \omega  & \omega ^{2}  & \cdots &\omega ^{n-1} \\
	\vdots & \vdots  & \vdots  & \ddots &\vdots \\
	1& \omega ^{n-1} & \omega ^{2\left ( n-1 \right )} &\cdots&\omega ^{\left ( n-1 \right )\left ( n-1 \right )}
	\end{bmatrix} \in \mathbb{C} ^{n\times n} $,
\end{center}
where $\omega =e^{-\frac{2\pi i}{n} }$ and $i$ is the imaginary unit. So we can learn that $\frac{\mathbf{F} _{n}}{\sqrt{n}}$ is orthogonal, i.e.,
\begin{equation}\label{2.1}
\mathbf{F}_{n}^{\ast}\mathbf{F}_{n}=\mathbf{F}_{n}\mathbf{F}_{n}^{\ast}=n\mathbf{I}_{n}.
\end{equation}
Thus $\mathbf{F}_{n}^{-1}=\frac{\mathbf{F} _{n}^{\ast}}{\sqrt{n}}$. 

Then, under the Matlab command fft, we can obtain $\hat{\mathcal{A}} \in \mathbb{C} ^{n_{1}\times n_{2}\times n_{3}   } $, the result of DFT on $\mathcal{A}$ along the 3-rd dimension for any $\mathcal{A} \in \mathbb{C} ^{n_{1}\times n_{2}\times n_{3}   } $, i.e.,
\begin{center}
	$\hat{\mathcal{A} } =\text{fft}\left ( \mathcal{A},\left [  \right ],3   \right ) $,
\end{center}
which is using the DFT on all the tubes of $\mathcal{A}$.

At the same time, by using the command inverse fft, we have
\begin{center}
	$\mathcal{A} =\text{ifft}\left ( \hat{\mathcal{A}},\left [  \right ],3   \right ) $.
\end{center}
Besides, $\bar{\mathbf{A} } \in \mathbb{C}^{n_{1}n_{3}\times n_{2}n_{3}   } $ is the block diagonal matrix whose $i$-th diagonal block is the $i$-th frontal slice $\hat{\mathbf{A}}^{\left ( i \right ) }$ of $\hat{\mathcal{A}}$, i.e.,
\begin{center}
	$\bar{\mathbf{A} } = \text{bdiag}\left(\hat{\mathcal{A}}\right) = \begin{bmatrix}
	\hat{\mathbf{A}}^{\left ( 1 \right ) }&  &   & \\
	& \hat{\mathbf{A}}^{\left ( 2 \right ) } &  & \\
	&  & \ddots  & \\
	&  &  &\hat{\mathbf{A}}^{\left ( n_{3} \right ) }
	\end{bmatrix}$.
\end{center}
where bidiag$\left ( \cdot  \right ) $ denotes the operator that maps the tensor $\hat{\mathcal{A}}$ to the block diagonal matrix $\bar{\mathbf{A}}$. Moreover, the block circulant matrix $\text{bcirc}\left(\mathcal{A}\right)\in \mathbb{C} ^{n_{1}n_{3}\times n_{2}n_{3}   } $ of $\mathcal{A}$ is given by
\begin{center}
	$\text{bcirc}\left (\mathcal{A}\right ) = \begin{bmatrix}
	\mathbf{A}^{\left ( 1 \right )}   &  \mathbf{A}^{\left ( n_{3} \right )} & \cdots  &  \mathbf{A}^{\left ( 2 \right )}\\
	\mathbf{A}^{\left ( 2 \right )} &  \mathbf{A}^{\left ( 1 \right )} & \cdots  &  \mathbf{A}^{\left ( 3 \right )}\\
	\vdots  & \ddots  & \ddots & \vdots\\
	\mathbf{A}^{\left ( n_{3} \right )} &  \mathbf{A}^{\left ( n_{3}-1 \right )} & \cdots & \mathbf{A}^{\left ( 1 \right )}
	\end{bmatrix}$	.
\end{center}
Then by using DFT, the block circulant matrix can be block diagonalized, i.e.,
\begin{equation}\label{2.2}
	\left ( \mathbf{F} _{{n}_{3}}\otimes \mathbf{I}_{{n}_{1}}  \right ) \cdot \text{bcirc}\left ( \mathcal{A}  \right )\cdot \left ( \mathbf{F} _{{n}_{3}}^{-1}\otimes \mathbf{I}_{{n}_{2}} \right )  =\bar{\mathbf{A}  } ,
\end{equation}
where $\otimes$ denotes the Kronecker product and $\frac{ \mathbf{F} _{{n}_{3}}^{-1}\otimes \mathbf{I}_{{n}_{2}}}{\sqrt{n_{3}}}$ is orthogonal. From \cite{LFCLLY2019}, $\hat{\mathbf{A}  }^{\left ( i \right ) }$ satisfies following property:
\begin{equation}\label{conj_property}
\begin{cases}
\hat{\mathbf{A}  }^{\left ( 1 \right ) }\in \mathbb{R}^{n_{1}\times n_{2}}  , & \\
\text{ conj }\left (\hat{\mathbf{A}  }^{\left ( i \right ) }  \right )=\hat{\mathbf{A}  }^{\left ( n_{3}-i+2 \right ) },  &i=2,...,\left[\frac{n_{3}+1}{2} \right]. 
\end{cases}
\end{equation}

Besides, based on (\ref{2.1}), we have the following imporitanct properties:
\begin{equation}\label{2.3}
\left \| \mathcal{A}  \right \| _{F } = \frac{1}{\sqrt{n_{3}} } \left \|\bar{\mathbf{A}}  \right \| _{F} ,
\end{equation}
\begin{equation}\label{2.4}
\left \langle \mathcal{A},\mathcal{B}   \right \rangle =\frac{1}{n_{3}} \left \langle \bar{\mathbf {A}},\bar{\mathbf {B}}  \right \rangle .
\end{equation}

\subsection{T-product, T-SVD and T-QR}
\begin{definition}\label{t-product}
\textbf{(T-product)\cite{KD2011}} Let $\mathcal{A}\in \mathbb{R}^{n_{1}\times n_{2}\times n_{3}}$ and $\mathcal{B}\in \mathbb{R}^{n_{2}\times l\times n_{3}}$. Then the t-product $\mathcal{A}\ast\mathcal{B}$ is the following tensor of size $n_{1}\times l\times n_{3}$:
\begin{equation}
\mathcal{A}\ast\mathcal{B}=\text{fold}\left(\text{bcirc}\left(\mathcal{A}\right)\cdot\text{unfold}\left(\mathcal{B}\right)\right).
\end{equation}
where $\text{unfold}\left( \cdot \right)$ and $\text{fold}\left( \cdot \right)$ are the unfold operator map and fold operator map of $\mathcal{A}$ respectively, i.e.,
\begin{center}
	$\text{unfold}\left(\mathcal{A}\right)=\begin{bmatrix}
	\mathbf{A}^{\left ( 1 \right ) }  \\
	\mathbf{A}^{\left ( 2 \right ) } \\
	\vdots  \\
	\mathbf{A}^{\left ( n_{3} \right ) } 
	\end{bmatrix}\in \mathbb{R} ^{n_{1}n_{3}\times n_{2}} ,\quad \text{fold}\left(\text{unfold}\left(\mathcal{A}\right)\right)=\mathcal{A}$
\end{center}
\end{definition}

From the above description and some details from \cite{KD2011}, we can know that the t-product is just matrix multiplication in the Fourier domain and replaces the multiplication between elements with circular convolution. Additionally, $\mathcal{C}=\mathcal{A}\ast\mathcal{B}$ is equivalent to $\bar{\mathcal{C}}=\bar{\mathcal{A}}\ast\bar{\mathcal{B}}$ due to \cite{LFCLLY2019}. Therefore, there are many similar properties between tensor multiplication and matrix multiplication, which provides us great confidence to explore properties of tensor by analogy with some properties of matrix.

\begin{definition}
\textbf{(Conjugate transpose)\cite{KD2011}} The conjugate transpose of a tensor $\mathcal{A}\in \mathbb{R}^{n_{1}\times n_{2}\times n_{3}}$ is the tensor $\mathcal{A}^{\ast}\in \mathbb{R}^{n_{2}\times n_{1}\times n_{3}}$ obtained by conjugate transposing each of the frontal slices and then reversing the order of transposed frontal slices 2 through $n_{3}$.
\end{definition}

\begin{definition}
	\textbf{(Zero tensor)}
	The zero tensor $\mathcal{O}\in \mathbb{R}^{n_{1}\times n_{2}\times n_{3}}$ is the tensor with all the elements equal to zero.
\end{definition}

\begin{definition}
\textbf{(Identity tensor)\cite{KD2011}} The identity tensor $\mathcal{I}\in \mathbb{R}^{n\times n\times n_{3}}$ is the tensor with the $n\times n$ identity matrix $\mathbf{I}_{n}$ as its first frontal slice and all other frontal slices being $\mathbf{O}_{n}$. 
\end{definition}

\begin{definition}
\textbf{(Orthogonal tensor)\cite{KD2011}} A tensor $\mathcal{Q}\in \mathbb{R}^{n\times n\times n_{3}}$ is orthogonal if it satisfies $\mathcal{Q}^{\ast}\ast\mathcal{Q}=\mathcal{Q}\ast\mathcal{Q}^{\ast}=\mathcal{I}$.
\end{definition}
Furthermore, the notion of partial orthogonality in \cite{KD2011} is defined and is similar to that a tall, thin matrix has orthogonal columns. Let $\mathcal{Q}\in \mathbb{R}^{n\times q\times n_{3}}$ is \textbf{partially orthogonal}, then $\mathcal{Q}^{\ast}\ast\mathcal{Q}$ can be defined and equivalent to $\mathcal{I}\in \mathbb{R}^{q\times q\times n_{3}}$. We can also obtain $\left \| \mathcal{Q}\ast \mathcal{A}   \right \| _{F} = \left \| \mathcal{A}  \right \| _{F} $ according to the properties of orthogonal tensors\cite{KD2011}. 

\begin{definition}
\textbf{(F-diagonal tensor)\cite{KD2011}} A tensor is called f-diagonal if each of its frontal slices is a diagonal matrix. 
\end{definition}

\begin{theorem}
\textbf{(T-SVD)\cite{LFCLLY2019}} Let $\mathcal{A}\in \mathbb{R}^{n_{1}\times n_{2}\times n_{3}}$. Then it can be factored as
\begin{equation}
\mathcal{A}=\mathcal{U}\ast\mathcal{S}\ast\mathcal{V}^{\ast}
\end{equation}
where $\mathcal{U}\in \mathbb{R}^{n_{1}\times n_{1}\times n_{3}}$ and $\mathcal{V}\in \mathbb{R}^{n_{2}\times n_{2}\times n_{3}}$ are orthogonal, and $\mathcal{S}\in \mathbb{R}^{n_{1}\times n_{2}\times n_{3}}$ is a f-diagonal tensor.
\end{theorem}

\begin{figure}[H]
	\centering
	\noindent\makebox[\textwidth][c] {
		\includegraphics[scale=0.4]{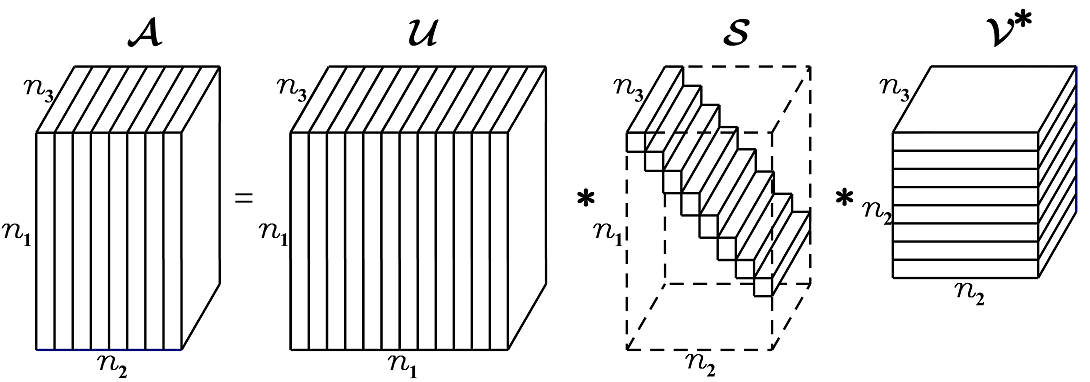} }
	\caption{The illustration of t-SVD of a $n_{1}\times n_{2}\times n_{3}$ tensor.}
	\label{t-SVD}
\end{figure}

The t-SVD is the framework of tensor tri-factorization widely used in tensor completion, as depicted in Figure \ref{t-SVD}. Its construction is similar to the SVD of matrix. T-SVD, thus, can be formalized as the sum of outer tensor product, i.e.,
\begin{center}
	$\mathcal{A}= {\textstyle \sum_{i=1}^{\min \left \{  n_{1},n_{2} \right \} }} \mathcal{U}\left (:,i,:  \right ) \ast\mathcal{S}\left (i,i,:  \right ) \ast\mathcal{V}\left (:,i,:  \right ) ^{\ast}$
\end{center}

\begin{theorem}\label{T-QR}
\textbf{(T-QR)\cite{KBHH2013}} Let $\mathcal{A}\in \mathbb{R}^{n_{1}\times n_{2}\times n_{3}}$. Then it can be factored as
\begin{equation}
\mathcal{A}=\mathcal{Q}\ast\mathcal{R}
\end{equation}
where $\mathcal{Q}\in \mathbb{R}^{n_{1}\times n_{1}\times n_{3}}$ is orthogonal, and $\mathcal{R}\in \mathbb{R}^{n_{1}\times n_{2}\times n_{3}}$ is analogous to the upper triangular matrix, as depicted in Figure \ref{t-QR}.
\end{theorem}

\begin{figure}[H]
	\centering
	\noindent\makebox[\textwidth][c] {
		\includegraphics[scale=0.4]{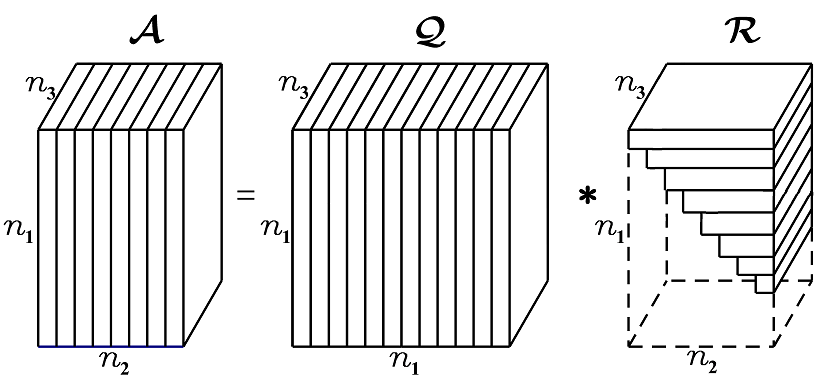} }
	\caption{The illustration of t-QR of a $n_{1}\times n_{2}\times n_{3}$ tensor.}
	\label{t-QR}
\end{figure}

\section{\large\bf Our method of tensor completion}

\subsection{A new method for third-order tensor decomposition}
Let $\mathbf{A}\in \mathbb{R}^{m\times n}$ be a real matrix; then, CSVD-QR \cite{LDYCJH2019} of $\mathbf{A}$ is as follows:
\begin{equation}
\mathbf{A} = \mathbf{L} \mathbf{D} \mathbf{R},
\end{equation}
where $\mathbf{L}\in \mathbb{R}^{m\times r}$, $\mathbf{R}\in \mathbb{R}^{r\times n}$ is the orthogonal matrix and $\mathbf{D}\in \mathbb{R}^{r\times r}$ is the lower triangular matrix whose $\mathbf{D}_{i,i}$ is the $i$th singular value of $\mathbf{A}$. 

Specific details of CSVD-QR can refer to the \cite{LDYCJH2019}
. Here is a brief description of the process of CSVD-QR:

First, let $\mathbf{L}_{k}=\text{eye}\left(m,r\right)$, $\mathbf{D}_{k}=\text{eye}\left(r,r\right)$ and $\mathbf{R}_{k}=\text{eye}\left(r,n\right)$, where $k$ is the $k$-th iteration.

Then, $\mathbf{L}_{k+1}$ is given by
\begin{equation}
\left[\mathbf{L}_{k+1},\sim \right]=\text{qr}\left(\mathbf{A}\mathbf{R}_{k}^{\ast}\right),
\end{equation}
where $\text{qr}\left(\cdot \right)$ is the function of economic QR decomposition in Matlab. $\mathbf{R}_{k+1}$ is given by
\begin{equation}
\left[\mathbf{R}_{k+1},\mathbf{T}\right]=\text{qr}\left(\mathbf{A}^{\ast}\mathbf{L}_{k+1}\right).
\end{equation}

Finally, $\mathbf{D}_{k+1}$ is updated by
\begin{equation}
\mathbf{D}_{k+1}=\mathbf{T}^{\ast}.
\end{equation}

Because of the simplicity and immediacy of CTSVD-QR, we are going to extending this decomposition to tensor decomposition. As for the function $\text{qr}\left(\cdot \right)$, we will use $\text{T-QR}\left(\cdot \right)$ in Theorem \ref{T-QR} for the subsequent optimization problem.

Suppose that $\mathcal{A}\in \mathbb{R}^{n_{1}\times n_{2}\times n_{3}}$ is a given real tensor. Then we are going to decompose $\mathcal{A}$ into $\mathcal{L}\in \mathbb{R}^{n_{1}\times r\times n_{3}}$, $\mathcal{D}\in \mathbb{R}^{r\times r\times n_{3}}$ and $\mathcal{R}\in \mathbb{R}^{r\times n_{2}\times n_{3}}$ via QR decomposition which just like t-SVD. By using such method, the largest $ r\in \left ( 0,\min \left \{ n_{1},n_{2} \right \}  \right ] $ singular values (tubes) and their associated singular vectors (of tubes) can be obtained quickly and directly. 
\begin{theorem}\label{CTSVD-QR}
\textbf{(CTSVD-QR).} For $\mathcal{A}\in \mathbb{R}^{n_{1}\times n_{2}\times n_{3}}$, it can be tri-factorized as 
\begin{equation}
\mathcal{A}=\mathcal{L}\ast \mathcal{D}\ast\mathcal{R},
\end{equation}
where $\mathcal{L}\in \mathbb{R}^{n_{1}\times r\times n_{3}}$,  $\mathcal{R}\in \mathbb{R}^{r\times n_{2}\times n_{3}}$ are orthogonal, and $\mathcal{D}\in \mathbb{R}^{r\times r\times n_{3}}$ is analogous to the $\mathcal{R}^{\ast}$ of Theorem \ref{T-QR}.
\end{theorem}
\begin{proof}
According to (\ref{2.2}), (\ref{conj_property}) and CSVD-QR. We can construct CSVD-QR of each $\hat{\mathbf{A}}^{\left ( i \right ) }$ by using the following method. For $i=1,...,\left[\frac{n_{3}+1}{2} \right]$, $\hat{\mathbf{A}}^{\left ( i \right ) }$ is decomposed into $\hat{\mathbf{L}}^{\left ( i \right ) }\hat{\mathbf{D}}^{\left ( i \right ) }\hat{\mathbf{R}}^{\left ( i \right ) }$ by CSVD-QR. Then for $i=\left[\frac{n_{3}+1}{2} \right]+1,...,n_{3}$, let $\hat{\mathbf{L}}^{\left ( i \right ) }=\text{conj}\left(\hat{\mathbf{L}}^{\left ( n_{3}-i+2 \right ) }\right)$, $\hat{\mathbf{D}}^{\left ( i \right ) }=\hat{\mathbf{D}}^{\left ( n_{3}-i+2 \right ) }$ and $\hat{\mathbf{V}}^{\left ( i \right ) }=\text{conj}\left(\hat{\mathbf{V}}^{\left ( n_{3}-i+2 \right ) }\right)$. Thus, we can obtain that $\hat{\mathbf{A}}^{\left ( i \right ) }=\hat{\mathbf{L}}^{\left ( i \right ) }\hat{\mathbf{D}}^{\left ( i \right ) }\hat{\mathbf{R}}^{\left ( i \right ) }$ for $i=1,...,n_{3}$ easily. Besides, $\bar{\mathbf{A}}=\bar{\mathbf{L}}\bar{\mathbf{D}}\bar{\mathbf{R}}$. 

Finally, on the basis of (\ref{2.2}), we can get the real block circulant matrices $\text{bcirc}\left(\mathcal{L}\right)$, $\text{bcirc}\left(\mathcal{D}\right)$ and $\text{bcirc}\left(\mathcal{R}\right)$ so that $\text{bcirc}\left(\mathcal{A}\right)$ is easy to obtain. Hence, the conclusion is obvious by folding up the result $\text{bcirc}\left(\mathcal{A}\right)$.
\end{proof}
\begin{algorithm}[H]
\caption{: Computing an Approximate T-SVD via QR Decomposition (CTSVD-QR)}
\label{algorithm1}
\begin{algorithmic}[1]
\Require  
$\mathcal{A}\in \mathbb{R}^{n_{1}\times n_{2}\times n_{3}}$.
\Ensure
CTSVD-QR components $\mathcal{L}$,$\mathcal{D}$ and $\mathcal{R}$ of $\mathcal{A}$.
\State Compute $\hat{\mathcal{A} } =\text{fft}\left ( \mathcal{A},\left [  \right ],3   \right ) $.
\State Compute each frontal slice of $\hat{\mathcal{L} }$, $\hat{\mathcal{D} }$ and $\hat{\mathcal{R} }$ from $\hat{\mathcal{A} }$ by
\For{$i=1,...,\left[\frac{n_{3}+1}{2} \right]$}
\State $\left [ \hat{\mathbf{L}}^{\left ( i \right ) }  ,\hat{\mathbf{D}}^{\left ( i \right ) },\hat{\mathbf{R}}^{\left ( i \right ) }  \right ]=\text{CSVD-QR}\left (\hat{\mathbf{A}}^{\left ( i \right ) }   \right )   $;
\EndFor
\For{$i=\left[\frac{n_{3}+1}{2} \right]+1,...,n_{3}$}
\State $\hat{\mathbf{L}}^{\left ( i \right ) }=\text{conj}\left(\hat{\mathbf{L}}^{\left ( n_{3}-i+2 \right ) }\right)$;
\State $\hat{\mathbf{D}}^{\left ( i \right ) }=\hat{\mathbf{D}}^{\left ( n_{3}-i+2 \right ) }$;
\State $\hat{\mathbf{V}}^{\left ( i \right ) }=\text{conj}\left(\hat{\mathbf{V}}^{\left ( n_{3}-i+2 \right ) }\right)$;
\EndFor
\State Compute $\mathcal{L} =\text{ifft}\left ( \hat{\mathcal{L}},\left [  \right ],3   \right ) $, $\mathcal{D} =\text{ifft}\left ( \hat{\mathcal{D}},\left [  \right ],3   \right ) $, and $\mathcal{R} =\text{ifft}\left ( \hat{\mathcal{R}},\left [  \right ],3   \right ) $.
\end{algorithmic}
\end{algorithm}

The result of Theorem \ref{CTSVD-QR} illustrates that a third-order tensor can be factorized into three components, which including two orthogonal tensors. Let $\mathcal{D}_{k}$ represent the $k$-th iteration of CTSVD-QR, $\mathcal{D}_{k}$ will converge to $\mathcal{S}$, which is the f-diagonal tensor in t-SVD. This drives us to use CTSVD-QR and some numberical tests will prove that is correct subsequently.
In the next subsection, we will use this method to initialize the data of the tensor completion problem.

\subsection{Tensor Completion based on $L_{2,1}$-Norm and Tensor QR Decomposition(TLNM-TQR)}
As we all know that the tensor completion problem can be modified as Problem (\ref{1.2}). If we wish to optimize tensor nuclear norm, we will use t-SVD to factorize $\mathcal{X}$ in general. However, solving Problem (\ref{1.2}) need a large number of iterations cause that t-SVD is a full rank decomposition.

Fortunately, we have CTSVD-QR to replace t-SVD so that the iterations can be reduced. In order to improve the speed of tensor completion, we also use tensor $L_{2,1}$-norm to replace the conventional tensor nuclear norm. Learning that $L_{2,1}$-norm of matrix is useful in feature extraction and has a good effect in solving the problem of matrix completion \cite{LDYCJH2019} and low-rank representation \cite{LLY2010}, here tensor $L_{2,1}$-norm of $\mathcal{X}\in\mathbb{R}^{n_{1}\times n_{2}\times n_{3}}$ is defined as follows:
\begin{equation}\label{L2,1-norm of tenosr}
\left \| \mathcal{X}  \right \| _{2,1}=
 {\textstyle \sum_{j=1}^{n_{2}}}\sqrt{ {\textstyle \sum_{i=1}^{n_{1}}} {\textstyle \sum_{k=1}^{n_{3}}} \mathcal{X}_{i,j,k}^{2}  }  .
\end{equation}
It is easy to know that the valid norm (\ref{L2,1-norm of tenosr}) satisfies the three norm conditions, including the triangle inequality, which can be proved by Cauchy-Schwarz inequality. In addition, we have an important formula as follows,
\begin{equation}\label{L2,1-norm of tenosr-other}
\left \| \mathcal{X}  \right \| _{2,1}={\textstyle \sum_{j=1}^{n_{2}}}\left \| \mathcal{X}\left ( :,j,: \right )   \right \|  _{F}.
\end{equation}

As for this tensor $L_{2,1}$-norm, we have the following properties.
%
	
\begin{theorem}\label{Tensor L2,1-norm}
As for CTSVD-QR of $\mathcal{A}\in \mathbb{R}^{n_{1}\times n_{2}\times n_{3}}$, $\mathcal{A}=\mathcal{L}\ast \mathcal{D}\ast\mathcal{R}$, we have the relationship between $\left \| \mathcal{D}  \right \| _{\ast}$ and $\left \|\mathcal{D}  \right \| _{2,1}$, i.e.,$\left \| \mathcal{D}  \right \| _{\ast}\le \left \|\mathcal{D}   \right \| _{2,1}$.
\end{theorem}
\begin{proof}
	The tensor $\mathcal{D}$ can be decomposed as follows:
	\begin{equation}\label{D to D^j 1}
	\mathcal{D}= {\textstyle \sum_{k=1}^{r}} \mathcal{D} ^{k},
	\end{equation}
	\begin{equation}\label{D to D^j 2}
	\mathcal{D} ^{k}_{i,j,t}=\begin{cases}
	\mathcal{D}_{i,k,t},&\left ( j=k \right )  \\
	0,&\left ( j\ne k \right ) 
	\end{cases},
	\end{equation}
	where $i,j=1,...,r$, $k=1,...,n_{3}$ and $\mathcal{D} ^{k}\in \mathbb{R}^{r\times r\times n_{3}}$. From (\ref{D to D^j 1}), we have 
	\begin{equation}\label{16-1}
	\left \| \mathcal{D}  \right \| _{\ast}=\left \| {\textstyle \sum_{k=1}^{r}} \mathcal{D} ^{k}  \right \| _{\ast}\le {\textstyle \sum_{k=1}^{r}}\left \|  \mathcal{D} ^{k}  \right \| _{\ast}.
	\end{equation}
	According to the property of tensor nuclear norm in \cite{LFCLLY2019} and the feature of matrix $\bar{\mathbf{D}^{k} }$, we have
	\begin{equation}\label{16}
	\left \| \mathcal{D}^{k}  \right \|_{\ast }=\frac{1}{n_{3}} \left \|  \bar{\mathbf{D}^{k} }\right \| _{\ast} =\frac{1}{n_{3}} \left \|  \bar{\mathbf{D}^{k} }\right \| _{F}  
	\end{equation}
	By combining (\ref{2.3}), (\ref{L2,1-norm of tenosr-other}), (\ref{16-1}) and (\ref{16}), here we can obtain
	\begin{equation}\label{18-1}
	{\textstyle \sum_{k=1}^{r}}\left \|  \mathcal{D} ^{k}  \right \| _{\ast}=\frac{1}{\sqrt{n_{3}}}\left \| \mathcal{D}  \right \| _{2,1} 
	\end{equation}
	Then, the conclusion can be proved via (\ref{16-1}) and (\ref{18-1})
\end{proof}

By the above description, the conclusion of Theorem \ref{Tensor L2,1-norm} shows that the nuclear norm of tensor $\mathcal{D}$ is clearly the lower bound on its tensor $L_{2,1}$-norm. This drives us to use the $L_{2,1}$-norm to Problem (\ref{1.2}) so that the problem can be rewritten as follows:
\begin{equation}\label{15}
\min_{\mathcal{D} } \left \| \mathcal{D}  \right \| _{2,1}   \quad \text{s.t.} \begin{cases}
 \mathcal{X}=\mathcal{L}\ast \mathcal{D}\ast \mathcal{R}    \\
 \mathcal{X}_{i,j,k}= \mathcal{M}_{i,j,k},\quad \left ( i,j,k \right )\in  \Omega
\end{cases}
\end{equation}
By the way, the optimization function of $L_{2,1}$-norm is convex, so the alternating direction method of multipliers (ADMM) \cite{LFCLLY2019} \cite{LFYL2018} can be used to settle Problem (\ref{15}). The augmented Lagrangian function of (\ref{15}) is
\begin{equation}
\text{L}\left ( \mathcal{L},\mathcal{D},\mathcal{R},\mathcal{Y},\mu    \right )   =\left \| \mathcal{D}  \right \| _{2,1}+\left \langle \mathcal{Y},\mathcal{X} - \mathcal{L}\ast\mathcal{D}\ast\mathcal{R}    \right \rangle  +\frac{\mu }{2}\left \|\mathcal{X} - \mathcal{L}\ast\mathcal{D}\ast\mathcal{R}  \right \|  _{F}^{2} ,
\end{equation}
where $\mathcal{Y} \in\mathbb{R}^{n_{1}\times n_{2}\times n_{3}} $ and $\mu > 0$. 

Different from the general ADMM method, whose each iteration follows a fixed stride, here we add CTSVD-QR to optimize the subproblem of ADMM to make the results of iteration tend to be optimal quickly so that our method will only need a few iterations to converge, which benefits from the fast and precise CTSVD-QR decomposition.

First step of all, $\mathcal{L}_{k+1}$ and $\mathcal{R}_{k+1}$ are updated by solving the problem as follows:
\begin{equation}\label{17}
\min_{\mathcal{L},\mathcal{R} } \left \| \left ( \mathcal{X}_{k}+\frac{\mathcal{Y}_{k} }{\mu _{k}} \right ) -\mathcal{L}\ast \mathcal{D}_{k}\ast \mathcal{R} \right \| ^{2 } _{F},
\end{equation}
where the subscript $k$ means the result of the $k$th iteration. Because the successive iteration of CTSVD-QR is special similar, we first use the ADMM method to initialize $\mathcal{L}_{k}$ and $\mathcal{R}_{k}$ in CTSVD-QR. So as for the minimization Problem (\ref{17}), $\mathcal{L}_{k+1}$ and $\mathcal{R}_{k+1}$ can be given by CTSVD-QR, i.e.,
\begin{equation}\label{18}
\left [ \mathcal{L}_{k+1},\sim    \right ] =\text{T-QR}\left(\mathcal{X}_{c}\ast \mathcal{R}^{\ast }_{k} \right);
\end{equation}
\begin{equation}\label{19}
\left [ \mathcal{R}_{k+1},\sim    \right ] =\text{T-QR}\left(\mathcal{X}_{c}^{\ast}\ast \mathcal{L}_{k+1} \right),
\end{equation}
where T-QR is the tensor QR decomposition in Theorem \ref{T-QR} and $\mathcal{X}_{c}=\mathcal{X}_{k}+\frac{\mathcal{Y}_{k} }{\mu _{k}}$. Because CTSVD-QR can get a precise result with only a few iterations. So when CTSVD-QR takes part in the iterative process of traditional ADMM, the speed of our tensor completion method can be improved, and the result of experiment is excellent by using such artful initialization.

The next step, we are going to update $\mathcal{D}$ and $\mathcal{X}$. For getting $\mathcal{D}_{k+1}$, there is a $L_{2,1}$-norm minimization problem as follows:
\begin{equation}\label{27}
\mathcal{D} _{k+1}=\text{arg}\min_{\mathcal{D} }  \frac{1}{\mu _{k}}\left \| \mathcal{D}  \right \|  _{2,1} +\frac{1}{2}\left \| \mathcal{D} -\mathcal{L}^{\ast }_{k+1} \ast \mathcal{X}_{c} \ast\mathcal{R}^{\ast }_{k+1}  \right \|^{2}_{F}  .
\end{equation}

Let $\mathcal{D} _{T}=\mathcal{L}^{\ast }_{k+1} \ast \mathcal{X}_{c} \ast\mathcal{R}^{\ast }_{k+1}$ and refer to (\ref{L2,1-norm of tenosr-other}), (\ref{D to D^j 1}) and (\ref{D to D^j 2}), we get the subproblem of Problem (\ref{27}) as follows,
\begin{equation}\label{problem D^j_k+1}
\mathcal{D}^{j} _{k+1}=\text{arg}\min_{\mathcal{D}^{j} }  \frac{1}{\mu _{k}}\left \| \mathcal{D}^{j}  \right \|  _{F} +\frac{1}{2}\left \| \mathcal{D}^{j} -\mathcal{D}^{j}_{T}  \right \|^{2}_{F}  ,
\end{equation}
where $j=1,...,r$. When we notice that (\ref{2.3}) and (\ref{16}), it is not difficult to gain
\begin{equation}
\left \| \mathcal{D}^{j}  \right \|  _{F}=\sqrt{n_{3}}\left \| \mathcal{D}^{j}  \right \|  _{\ast},
\end{equation}
so that the Problem (\ref{problem D^j_k+1}) can be equivalent to
\begin{equation}\label{problem D^j_k+1*}
\mathcal{D} _{k+1}\left(:,j,:\right)=\text{arg}\min_{\mathcal{D}\left(:,j,:\right) }  \frac{1}{\mu _{k}}\left \| \mathcal{D}\left(:,j,:\right)  \right \|  _{\ast} +\frac{1}{2}\left \| \mathcal{D}\left(:,j,:\right) -\mathcal{D}_{T}\left(:,j,:\right)  \right \|^{2}_{F}  .
\end{equation}
Then the optimal solver of Problem (\ref{problem D^j_k+1*}) can be given by the Theorem 4.2 of \cite{LFCLLY2019}, i.e.,
\begin{equation}\label{D_k+1_1}
\mathcal{D} _{k+1}\left ( :,j,: \right ) =\mathcal{L}^{0}\ast\mathcal{D}^{0}_{\frac{1}{\mu } }\ast\mathcal{R}^{0},
\end{equation}
\begin{equation}\label{D_k+1_2}
\mathcal{D}^{0}_{\frac{1}{\mu } }=\text{ifft}\left(\max\left \{\hat{\mathcal{D}^{0}}-\frac{1}{\mu },0\right \} ,[],3\right),
\end{equation}
where $\mathcal{L}^{0}$, $\mathcal{D}^{0}$ and  $\mathcal{R}^{0}$ is CTSVD-QR of $\mathcal{D}_{T}\left(:,j,:\right)$. 

On account of
\begin{center}
 $\mathcal{D}^{0}\left(:,j,t\right)=\left\|\mathcal{D}_{T}\left(:,j,t\right)\right\|_{F}$,
\end{center} 
where $t=1,...,n_{3}$, (\ref{D_k+1_1}) and (\ref{D_k+1_2}) can be modified as
\begin{equation}\label{D_k+1_3}
\mathcal{D}_{k+1}\left(:,j,t\right)=\text{ifft}\left(\frac{\max \left \{ \left \| \hat{ \mathcal{D}_{T}}\left(:,j,t\right) \right \|_{F} -\frac{1}{\mu },0 \right \}  }{\left \| \hat{ \mathcal{D}_{T}}\left(:,j,t\right) \right \|_{F}}\ast  \hat{ \mathcal{D}_{T}}\left(:,j,t\right) ,[],3\right) .
\end{equation}
By such optimization method, we can build a for loop in Matlab to solve the Problem (\ref{27}) to reduce the memory load on the computer.

As for $\mathcal{X}_{k+1}$, it is updated by
\begin{equation}\label{24}
\mathcal{X}_{k+1}=\mathcal{L}_{k+1}\ast\mathcal{D}_{k+1}\ast\mathcal{R}_{k+1}-\left(\mathcal{L}_{k+1}\ast\mathcal{D}_{k+1}\ast\mathcal{R}_{k+1}\right)_{\Omega}+\mathcal{M}_{\Omega},
\end{equation}
where $\Omega$ is defined in (\ref{1.1}).

The final step, $\mathcal{Y}_{k+1}$ and $\mu_{k+1}$ are updated as follows:
\begin{equation}\label{25}
\mathcal{Y}_{k+1}=\mathcal{Y}_{k}+\mu_{k}\left(\mathcal{X}_{k+1}-\mathcal{L}_{k+1}\ast\mathcal{D}_{k+1}\ast\mathcal{R}_{k+1}\right),
\end{equation}
\begin{equation}\label{26}
\mu_{k+1}=\rho \mu_{k}.
\end{equation}
where $\rho\ge 1$.

\begin{algorithm}[H]
	\caption{: Tensor $L_{2,1}$-Norm Minimization Method Based on tensor QR decomposition (TLNM-TQR)}
	\label{algorithm2}
	\begin{algorithmic}[1]
		\Require  
		A real incomplete tensor $\mathcal{M}\in \mathbb{R}^{n_{1}\times n_{2}\times n_{3}}$;
		$\Omega $, the set of observed locations. 
		\Ensure
		$\mathcal{X}$, the recovery result of $\mathcal{M}$  .
		\State Initialize $r>0$, $k=0$, $t>0$, $\mu>0$, $\rho>0$, $\epsilon >0$, $\mathcal{Y}=\mathcal{O}\in \mathbb{R}^{n_{1}\times n_{2}\times n_{3}}$, $\mathcal{L}_{0}=\mathcal{I}\in \mathbb{R}^{n_{1}\times r\times n_{3}}$, $\mathcal{D}_{0}=\mathcal{I}\in \mathbb{R}^{r\times r\times n_{3}}$, $\mathcal{R}_{0}=\mathcal{I}\in \mathbb{R}^{r\times n_{2}\times n_{3}}$, $\mathcal{X}_{0}=\mathcal{M}_{\Omega }$.
		\While{$\left \|\mathcal{L}_{k}\ast\mathcal{D}_{k}\ast\mathcal{R}_{k}-\mathcal{X}_{k}\right \| ^{2}_{F}\ge \epsilon \quad \text{and}\quad k>t $}
		
		\State$\left [ \mathcal{L}_{k+1},\sim    \right ] =\text{T-QR}\left(\left(\mathcal{X}_{k}+\frac{\mathcal{Y}_{k} }{\mu _{k}}\right)\ast \mathcal{R}^{\ast }_{k} \right)$;
		\State $\left [ \mathcal{R}_{k+1},\mathcal{D}_{T}^{\ast}    \right ] =\text{T-QR}\left(\left(\mathcal{X}_{k}+\frac{\mathcal{Y}_{k} }{\mu _{k}}\right)^{\ast}\ast \mathcal{L}_{k+1} \right)$;
		\For{$t=1,...,n_{3}$}
		\For{$j=1,...,r$}
        \State $\mathcal{D}_{k+1}\left(:,j,t\right)=\text{ifft}\left(\frac{\max \left \{ \left \| \hat{ \mathcal{D}_{T}}\left(:,j,t\right) \right \|_{F} -\frac{1}{\mu },0 \right \}  }{\left \| \hat{ \mathcal{D}_{T}}\left(:,j,t\right) \right \|_{F}}\ast  \hat{ \mathcal{D}_{T}}\left(:,j,t\right) ,[],3\right) $;
        \EndFor
        \EndFor
        \State $\mathcal{X}_{k+1}=\mathcal{L}_{k+1}\ast\mathcal{D}_{k+1}\ast\mathcal{R}_{k+1}-\left(\mathcal{L}_{k+1}\ast\mathcal{D}_{k+1}\ast\mathcal{R}_{k+1}\right)_{\Omega}+\mathcal{M}_{\Omega}$;
        \State $\mathcal{Y}_{k+1}=\mathcal{Y}_{k}+\mu_{k}\left(\mathcal{X}_{k+1}-\mathcal{L}_{k+1}\ast\mathcal{D}_{k+1}\ast\mathcal{R}_{k+1}\right)$;
        \State $\mu_{k+1}=\rho \mu_{k}$;

		
		\EndWhile
 \State \Return $\mathcal{X}=\mathcal{X}_{k}$.
	\end{algorithmic}
\end{algorithm}
ADMM is a gradient descent method and has the convergence in \cite{LFCLLY2019}. Our TLNM-TQR method can be considered as a modified-ADMM with t-QR added, whose stride is set up to be more flexible, so that TLNM-TQR can finally attain the global optimal solution with tiny error quickly.


\section{Numerical Tests And Applications}
To demonstrate the proposed methods are practicable, several comparative experiments adopting synthetic and real-world data, such as videos and color images, are performed, and its results are compared with root-mean-square error (RMSE), i.e, $\text{RMSE}=\sqrt{\frac{\left \| \mathcal{X}-\mathcal{Y}   \right \| _{F} ^{2}}{n_{1}n_{2}n_{3}}  }$, where $\mathcal{X}, \mathcal{Y}\in\mathbb{R}^{n_{1}\times n_{2}\times n_{3}} $ are the result date and original data respectively.

All our methods have been implemented for tensor completion on MATLAB 2018a platform equipped with an Intel Core i3-3110M CPU and 8GB of RAM.
\subsection{Convergence of CTSVD-QR}
In this section, synthetic tensor $\mathcal{X}$ is constructed as follows:
\begin{equation}
\mathcal{X} =\mathcal{M}^{m\times r_{1} \times p}_{1}  \ast \mathcal{M}^{r_{1} \times n \times p}_{2} ,
\end{equation}
\begin{equation}
\mathcal{M}^{m\times r_{1} \times p}_{1}=\text{randn}\left(m,r_{1},p\right),
\end{equation}
\begin{equation}
\mathcal{M}^{r_{1} \times n \times p}_{2}=\text{randn}\left(r_{1},n,p\right),
\end{equation}
where $r_{1}\in \left[1,n\right]$ is the tubal-rank of $\mathcal{X}$. CTSVD-QR is tested on $\mathcal{X}$ so as to prove this tensor tri-factorization is feasible. 
\begin{figure}[H]
	\centering
	\includegraphics[scale=0.6]{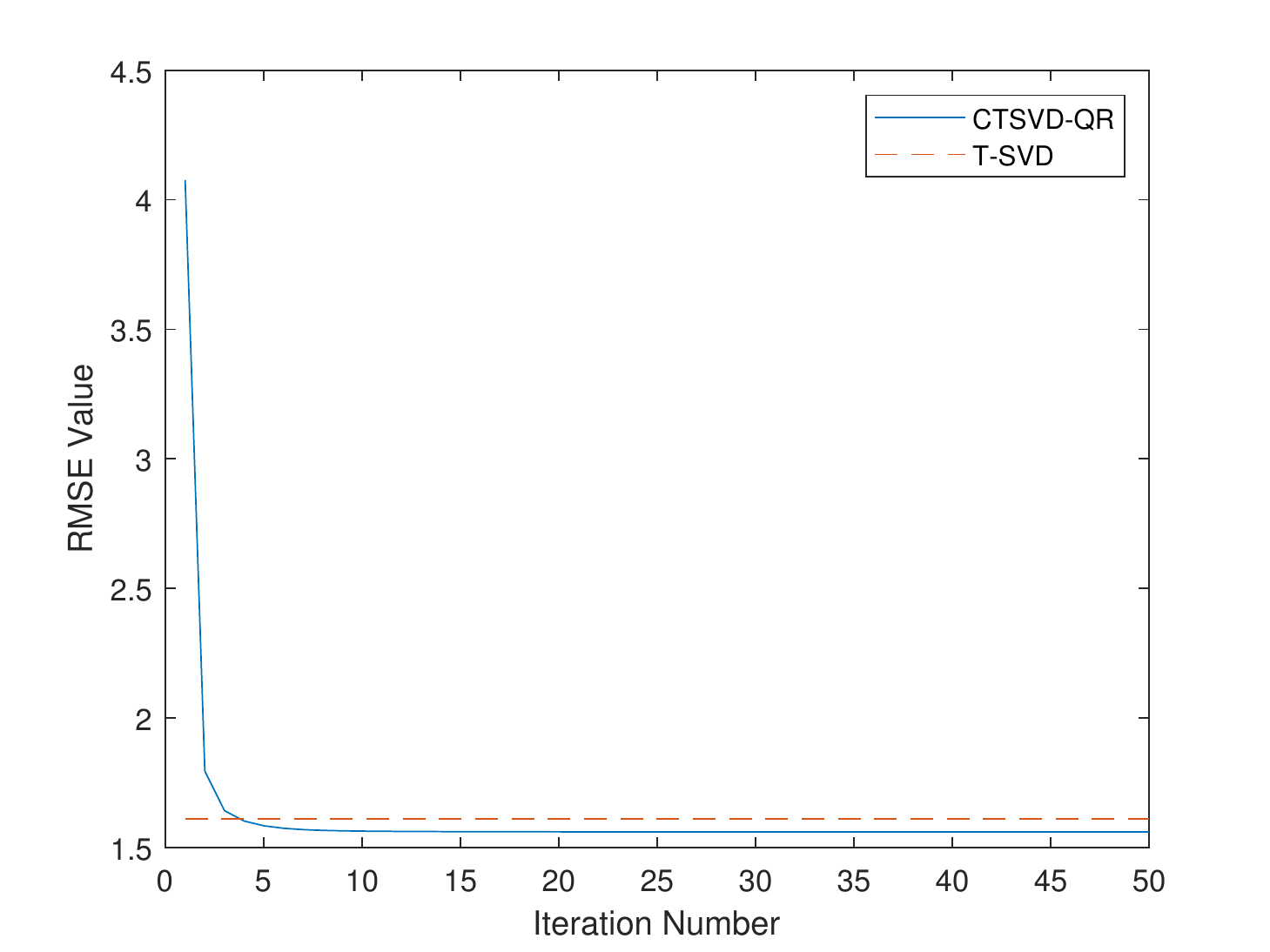}
	\caption{RMSE Value of CTSVD-QR}
	\label{Convergence of CTSVD-QR}
\end{figure}
First, let $m=n=300, p=3, r_{1}=250$ and $r=200$, where $r$ is the largest $r$ singular values (tubes). Then decomposing $\mathcal{X}$ by CTSVD-QR and T-SVD respectively, we will obtain the tri-factorization of $\mathcal{X}$. Reconstructing the new tensor $\mathcal{X}$ via the result of factorization, we have the Figure \ref{Convergence of CTSVD-QR} finally.

The Figure \ref{Convergence of CTSVD-QR} shows us that CTSVD-QR will converge to a definite value at the time when the iteration number increases continually. And our accuracy of CTSVD-QR is close to T-SVD and is better than it. Besides, when their RMSE Value are nearly equal, the running times of CTSVD-QR and T-SVD are 0.3432 and 0.6864 s.

\begin{figure}[H]
	\centering
	\includegraphics[scale=0.6]{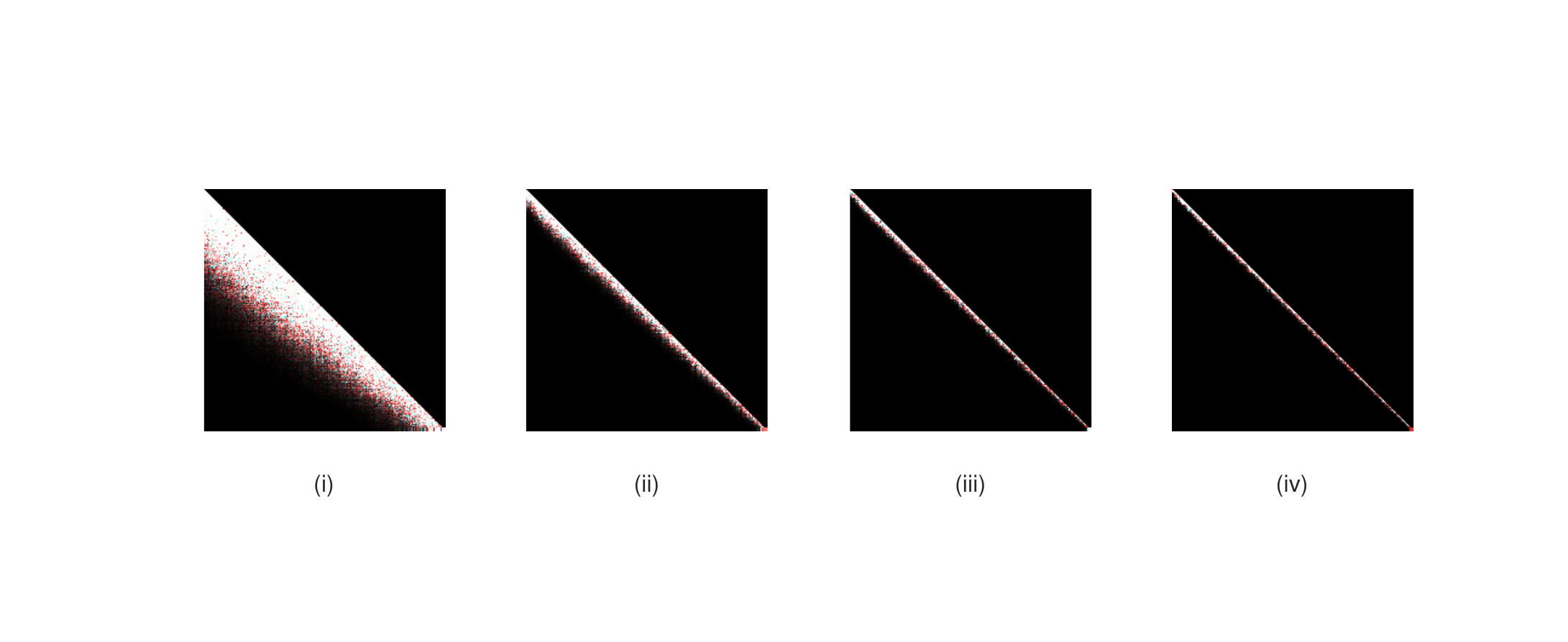}
	\caption{Convergence procedure of $\mathcal{D}_{k}$ in CTSVD-QR. $\left(\text{i}\right) k=5$. $\left(\text{ii}\right) k=20$. $\left(\text{iii}\right) k=40$. $\left(\text{iv}\right) k=60$.}
	\label{Convergence of D}
\end{figure}
In addition, we chose $\mathcal{D}_{5}$, $\mathcal{D}_{20}$, $\mathcal{D}_{40}$ and $\mathcal{D}_{60}$ to explain that $\mathcal{D}_{k}$ will converge to an f-diagonal tensor which is mentioned above. The convergence can be shown by ploting 100$\mathcal{D}_{k}\left(k=5, 20, 40, 60\right)$, as shown in Figure \ref{Convergence of D}.

\subsection{Experimental Result of TLNM-TQR}
In this section, TLNM-TQR is tested using video and color images. The RMSE value and running time are compared with three state-of-the art algorithms such as the ADMM with t-SVD as subroutine (ADMM-t-SVD) \cite{ZA2017}, HoMP \cite{YMS2015} and ELRAP4TS \cite{X2020}. For the sake of fairness, the parameters of ADMM-t-SVD, HoMP and ELRAP4TS are set to the optimal values. 

\subsubsection{Video Recovery}
Here we use a black and white gray-scale basketball video, which depicts 1.6 seconds of the game and contains 40 digital images in AVI format. Each frame of this video holds $144\times256$ pixels of black and white images so that it can be considered as a three-dimensional tensor $\mathcal{X}\in \mathbb{R}^{144\times 256\times 40}$. If we retain $50\%$ of given pixels as known entries in $\Omega$, the problem of video recovery will be a tensor completion problem. Thus, we can slove this problem with our TLNM-TQR algorithm.

\begin{figure}[H]
  \begin{center}
			\includegraphics[scale=0.45]{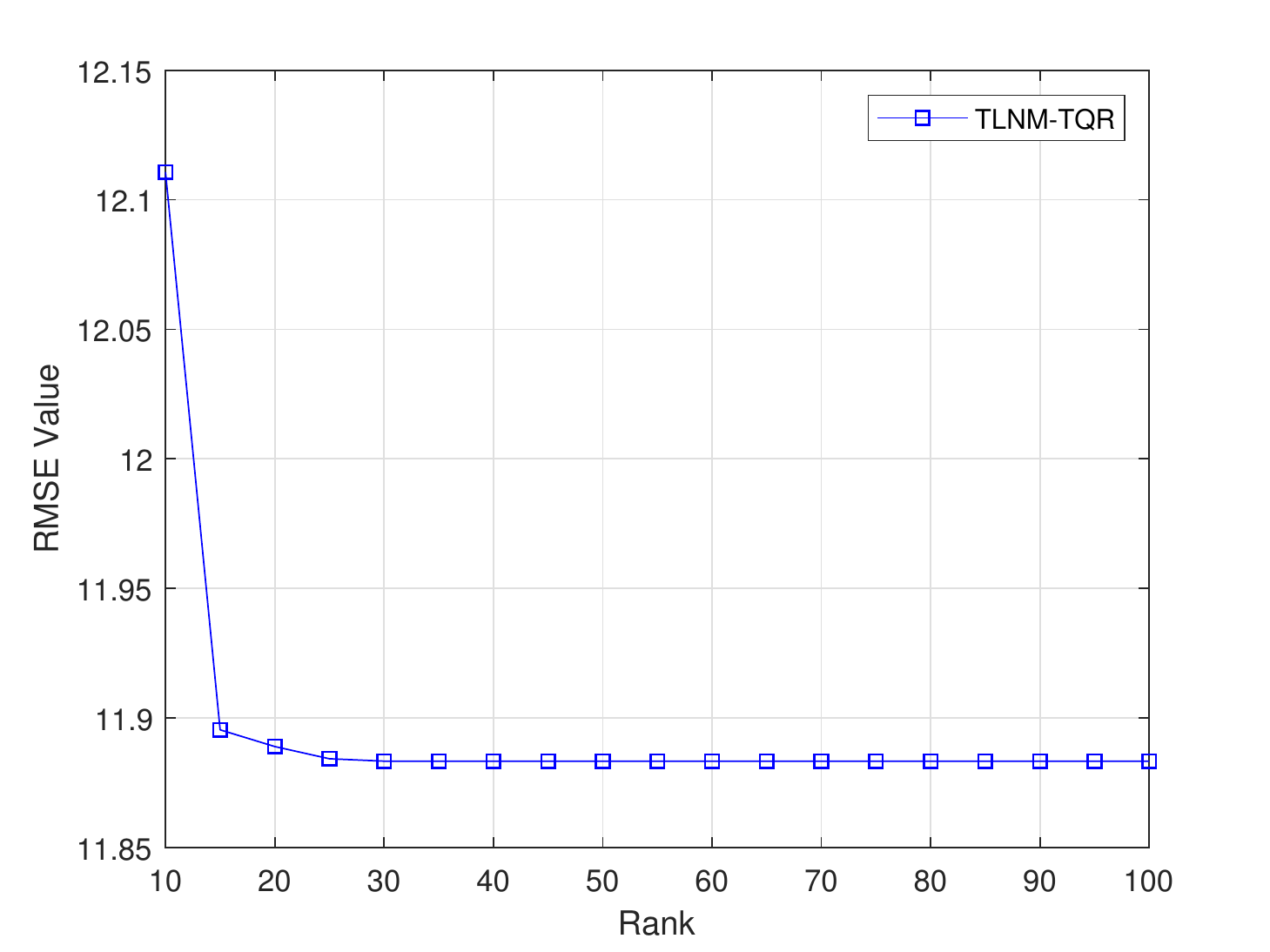}  
			\includegraphics[scale=0.45]{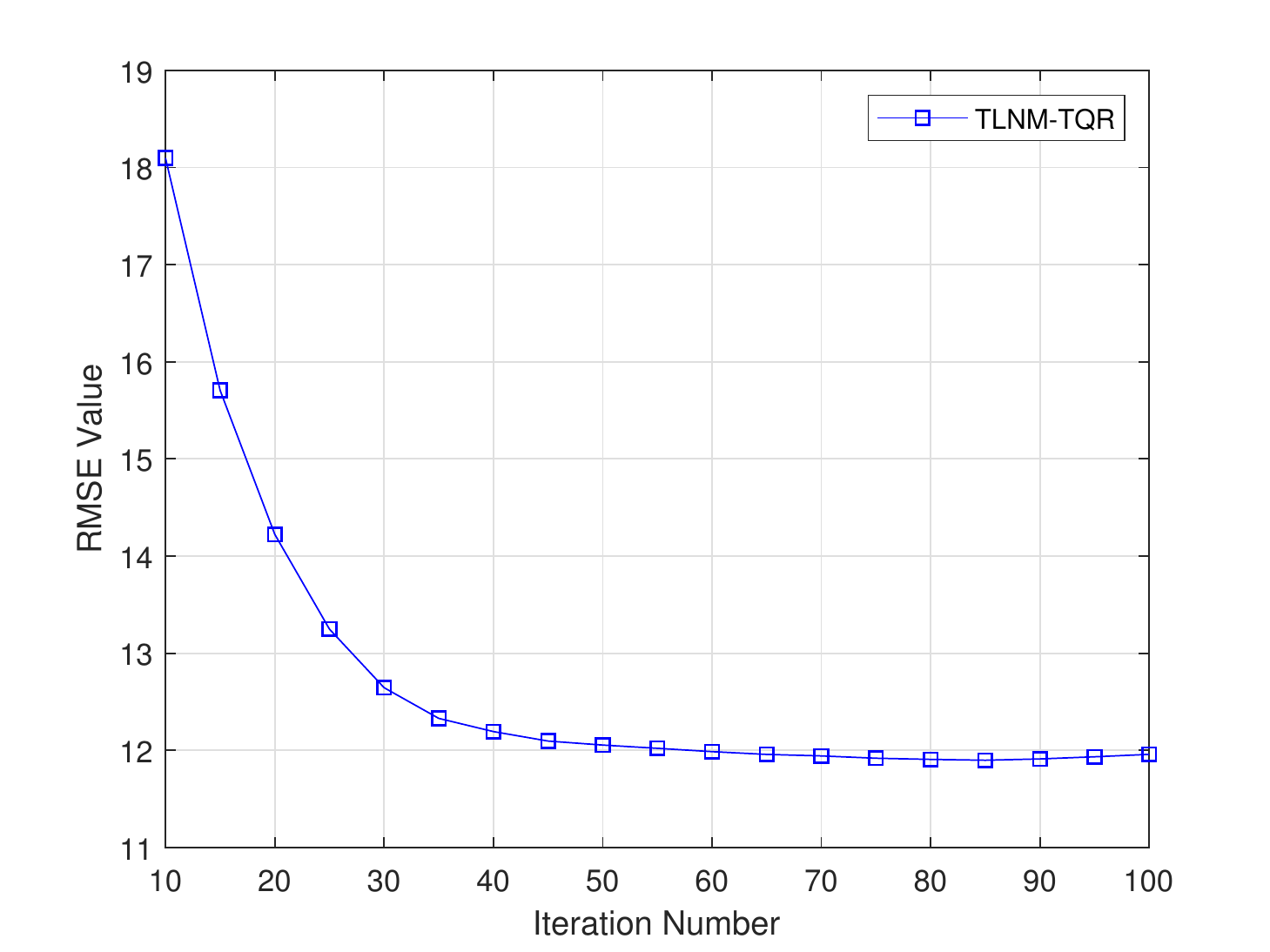}  
	\caption{Linear convergence of TLNM-TQR with 100 iterations $\left(\text{left column} \right)$ and $r=11 \left(\text{right column}\right)$} 
	\label{Convergence of TLNM-TQR}
   \end{center}
\end{figure}

To verify that TLNM-TQR is convergent, we chose a video with $50\%$ miss rate to test our algorithm. When all the parameters are fixed, the required inputs in TLNM-TQR are the iteration number and tubal-rank of result. Hence we fixed iteration number and tubal-rank respectively in Figure \ref{Convergence of TLNM-TQR}. It is easy to see that our mehtod has an excellent convergence accuracy and only needs a few iterations to converge so that our algorithm is fast and precise. This also gives us confidence to transcend other algorithms that have been popular in recent years.

\begin{figure}[H]
	\centering
	\includegraphics[scale=0.38]{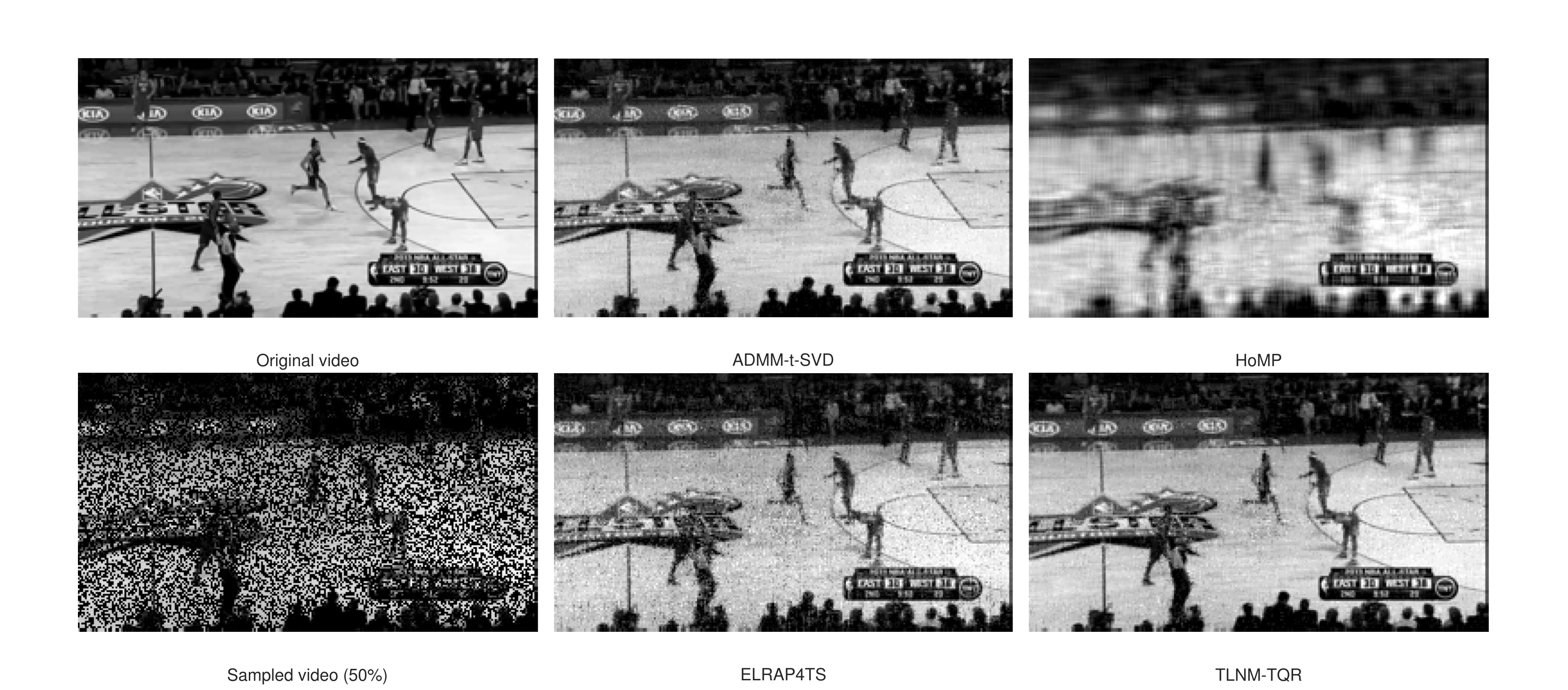} 
	\caption{The 30th frame of completion result for a basketball video.}
	\label{The 30th frame of completion result for a basketball video}
\end{figure}

In order to achieve the optimal effect of the algorithms, we set the parameter $\lambda$ of ADMM-t-SVD to $\lambda =\frac{1}{\sqrt{3\max \left ( n_{1},n_{2} \right ) } } $ and adjust its iteration to be 20. For HoMP and ELRAP4TS, we set $r=100$ and $r=70$ respectively. Besides, the $s$ in ELRAP4TS is set to $s=3$. Then the input of our TLNM-TQR method is $r=11$, and its parameter $\mu$ and $\rho$ are adjusted to be $\mu=10^{-2}$ and $\rho=1.5$ respectively. From Figure \ref{The 30th frame of completion result for a basketball video}, our algorithm is stable while the miss rate is $50\%$ and the detail comparison datas are listed in Table \ref{Runnig time and RMSE of tensor completion result on the basketball video}.
 
\begin{table}[H]
	\caption{Runnig time and RMSE of tensor completion result on the basketball video}
	\label{Runnig time and RMSE of tensor completion result on the basketball video}
	\begin{center} 
    \begin{tabular}{lcc}
	    \toprule
		Completion Approach & Running Time (seconds) & RMSE Value   \\
		\midrule
		ADMM-t-SVD          & 53.5863      & \textbf{11.7232} \\
		HoMP                & 47.4555      & 29.7052 \\
		ELRAP4TS            & 16.3333      & 20.8765 \\
		TLNM-TQR            & \textbf{11.0136}      & 12.2161 \\
		\bottomrule
    \end{tabular}
  \end{center}
\end{table}

From these results, we can see that the speed of our algorithm is faster than others, and its accuracy is almost close to ADMM-t-SVD which can also be seen in Figure \ref{The 30th frame of completion result for a basketball video}. More than that, we will explore the performance of our algorithm with a variety of specified missing ratio from $30\%$ to $90\%$. All algorithms are going to be tested 50 times to reduce contingency. 

\begin{figure}[H]
   	\begin{center}
			\includegraphics[scale=0.45]{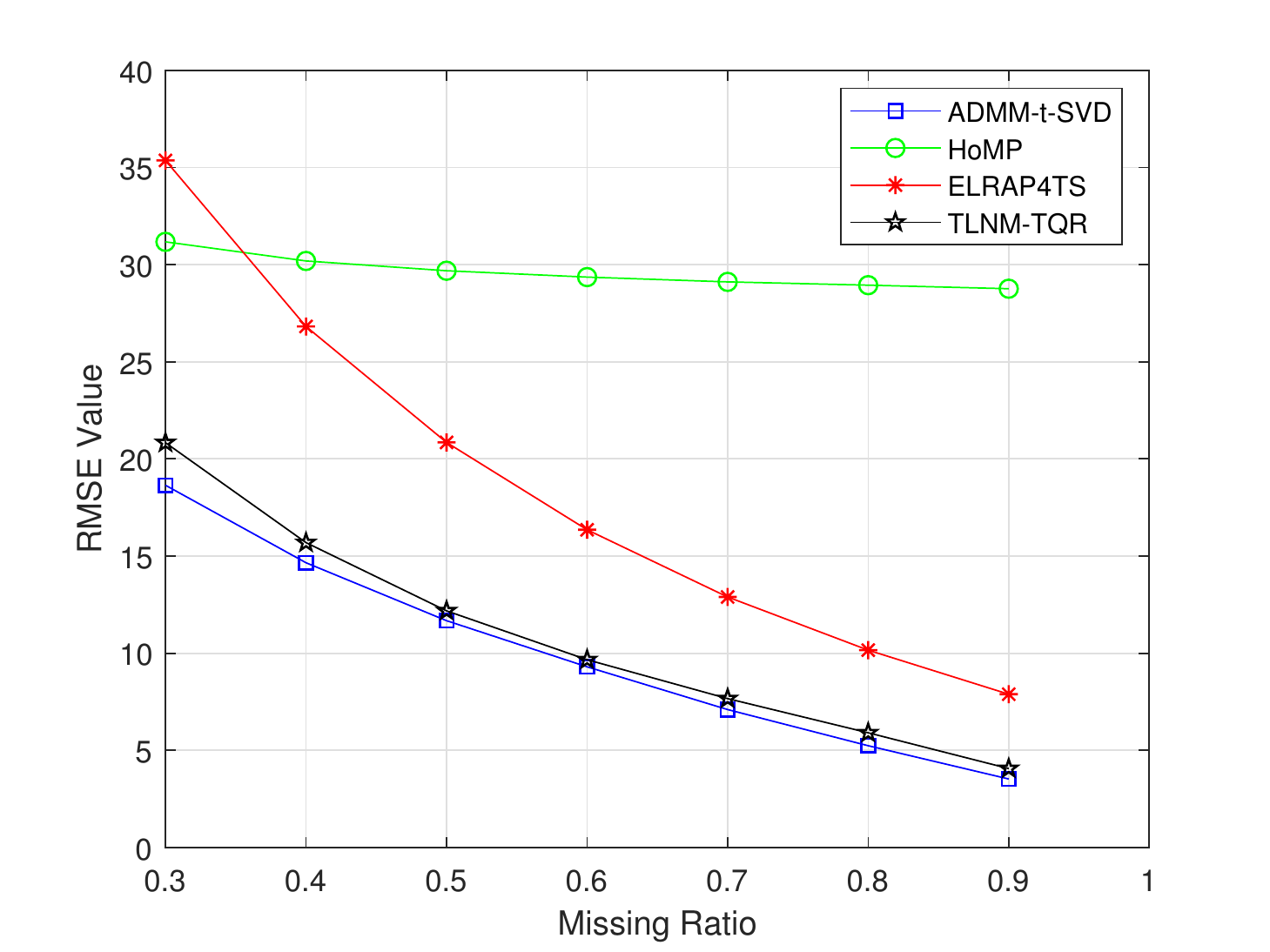}  
			\includegraphics[scale=0.45]{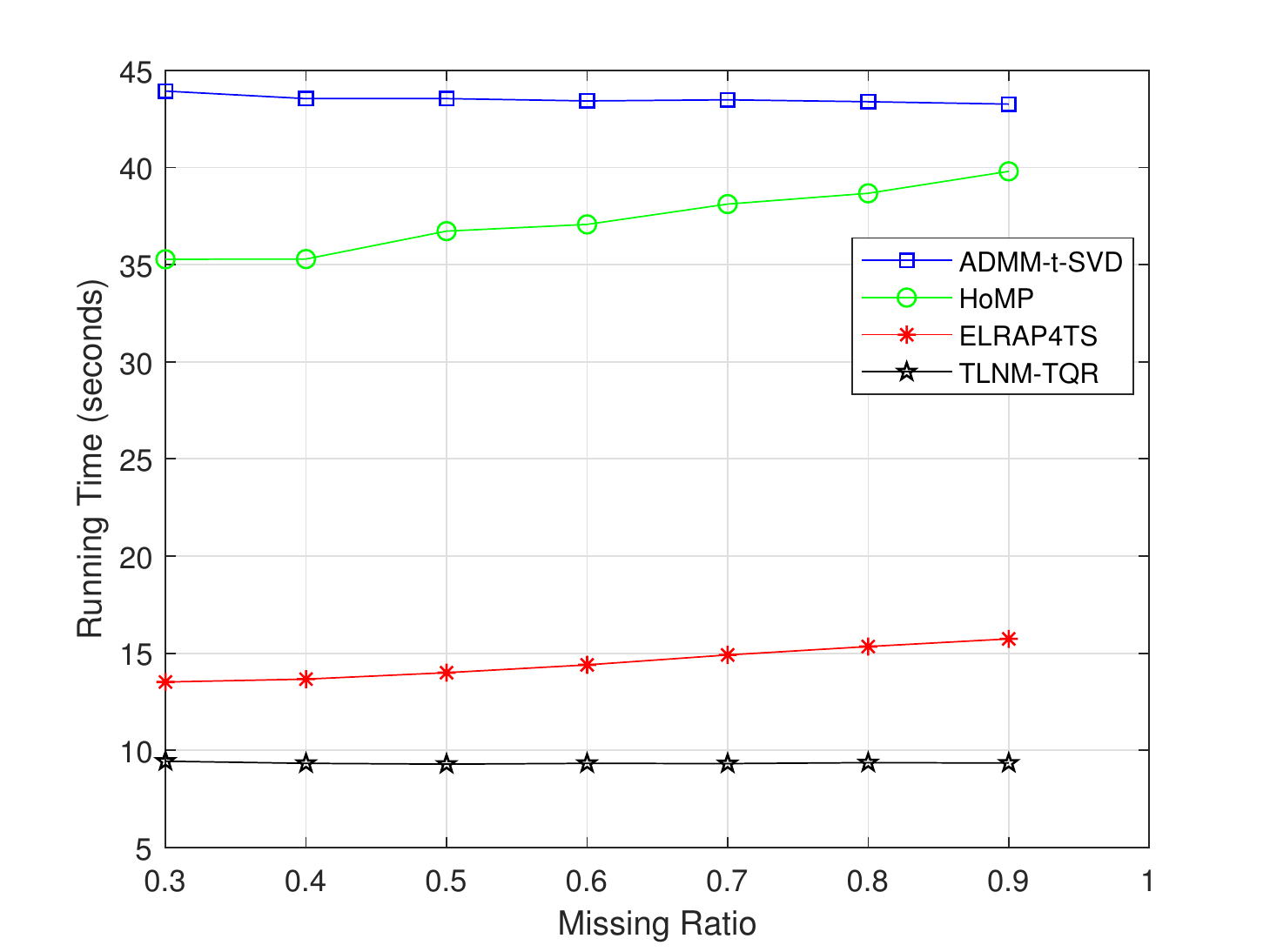}  
	\caption{RMSE Value $\left(\text{left column}\right)$ and Running Time $\left(\text{right column}\right)$ of the four algorithms with different missing ratio.} 
	\label{comparative result of TLNM-TQR}
	\end{center}
\end{figure}

As can be seen from Figure \ref{comparative result of TLNM-TQR}, TLNM-TQR is as precise as ADMM-t-SVD and more precise than others for each specified missing ratio. Besides, TLNM-TQR is also the fast method and is very steady on the speed side. The reason is that the iteration of the algorithm is reduced by CTSVD-QR, a fast and stable decomposition, in solving the Problem (\ref{17}). Furthermore, tensor $L_{2,1}$-norm helps us carry out optimized work more easily.

\begin{figure}[H]
	\centering
			\includegraphics[scale=0.45]{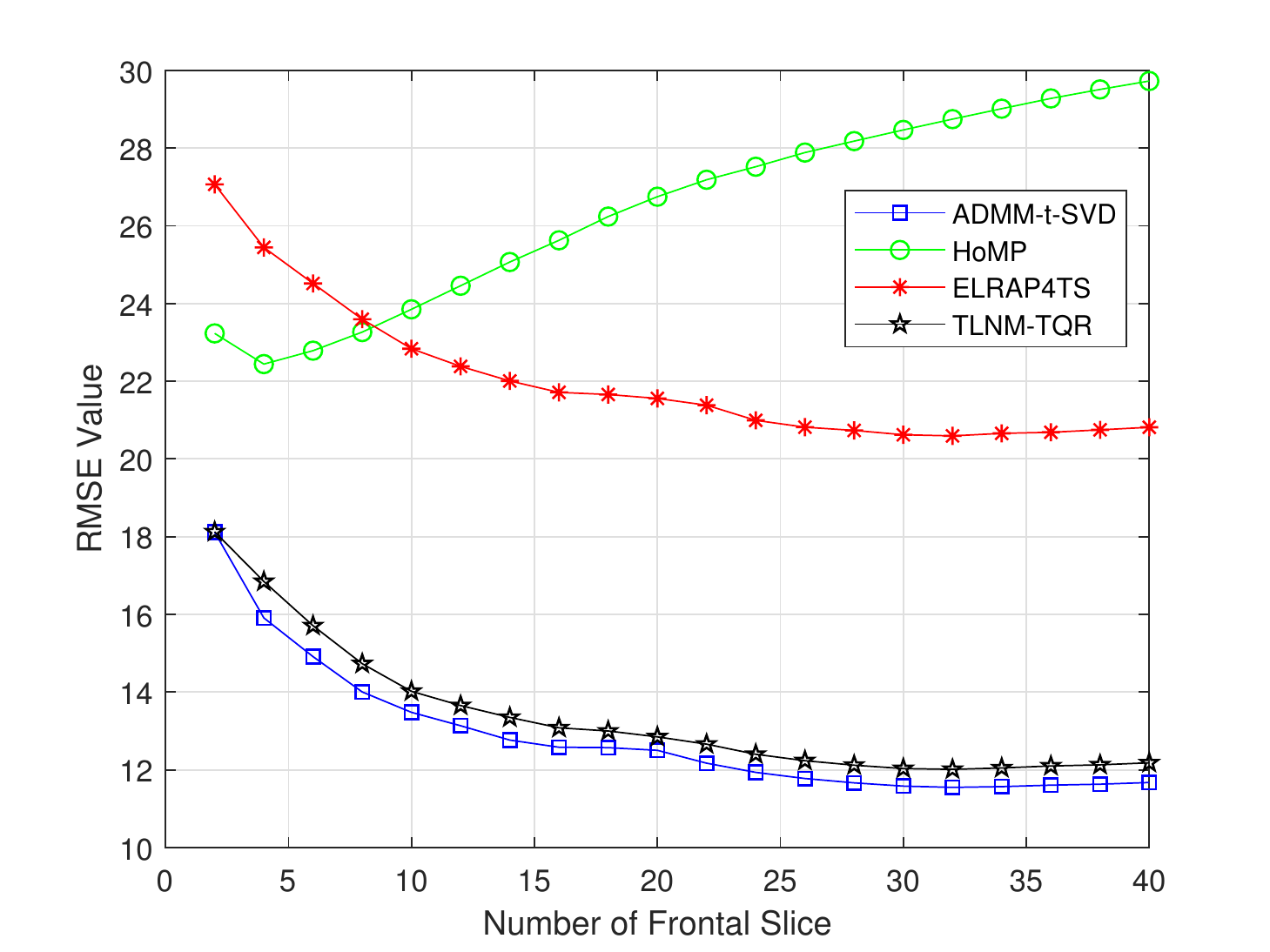}  
			\includegraphics[scale=0.45]{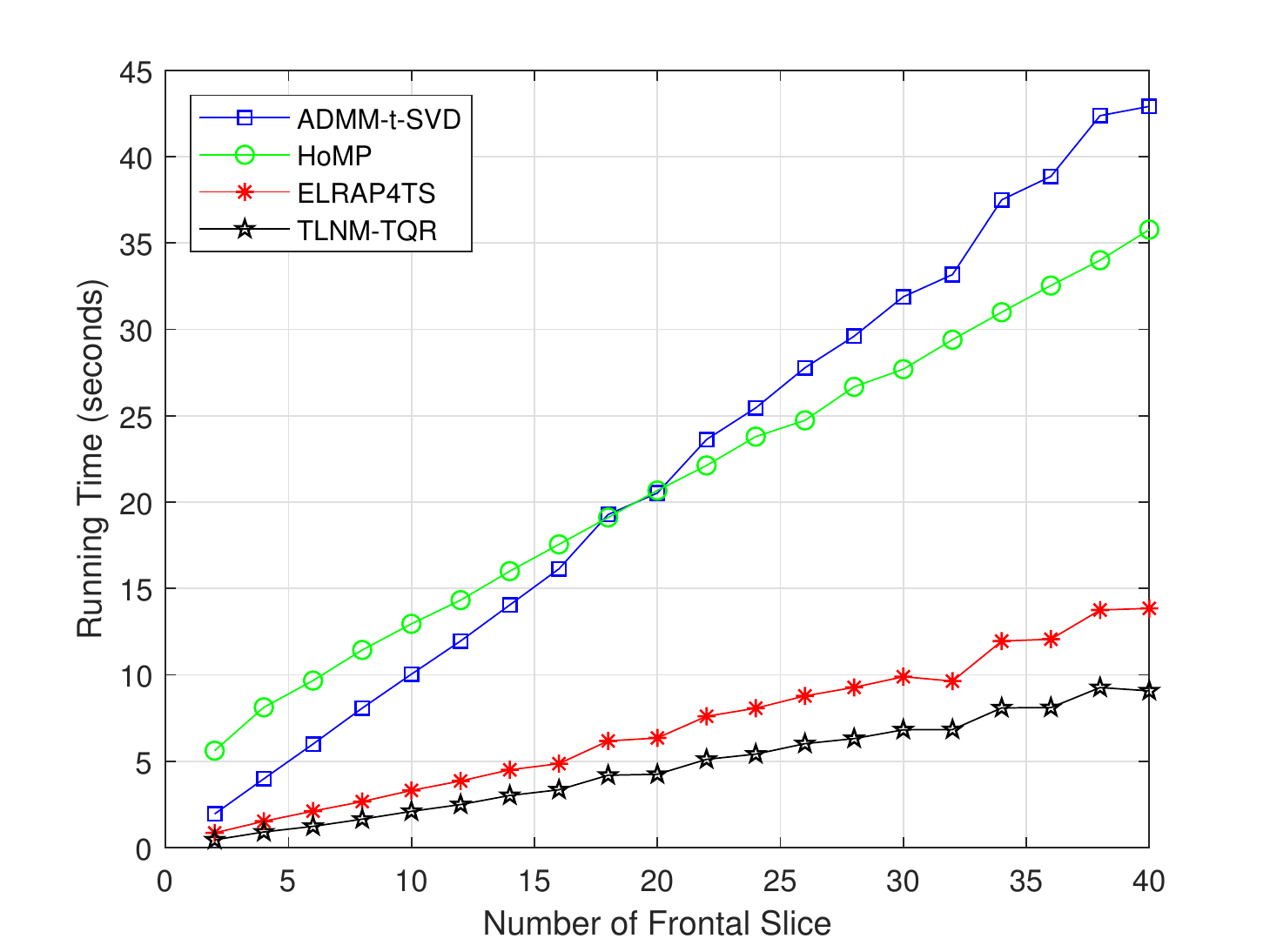}  
	
	\caption{RMSE Value $\left(\text{left column}\right)$ and Running Time $\left(\text{right column}\right)$ of the four algorithms with different number of frontal slice.} 
	\label{frontal comparative result of TLNM-TQR}
\end{figure}

The next experiment, we are going to set a variety of frontal slice $\left( \text{video frame} \right)$ number from 2 to 40 to give recovery performance comparison of four algorithms whose experiment setup is the same as the above-mentioned. As shown in Figure \ref{frontal comparative result of TLNM-TQR}, the RMSE result of TLNM-TQR is not only almost close to ADMM-t-SVD but superior to others when the number of frontal slice is increasing ceaselessly. Besides, the running time line chart on the right shows that the rate of rise of TLNM-TQR is lower than others, which means that our algorithm will have a huge advantage while dealing with more complex tensor structures and datas.

\subsubsection{Color Image Recovery}
In this part, we use a $n_{1}\times n_{2}$ sizes color image as a $n_{1}\times n_{2}\times 3$ tensor to test the performance of our algorithm and compare it with other state-of-the-art algorithms with the same parameters as previously set as well. All of the color images are from the Berkeley Segmentation Dataset \cite{MFTM2001} and thier size are $481\times 321$ or $321\times 481$. 

\begin{figure}[H]
	\centering
			\includegraphics[scale=0.45]{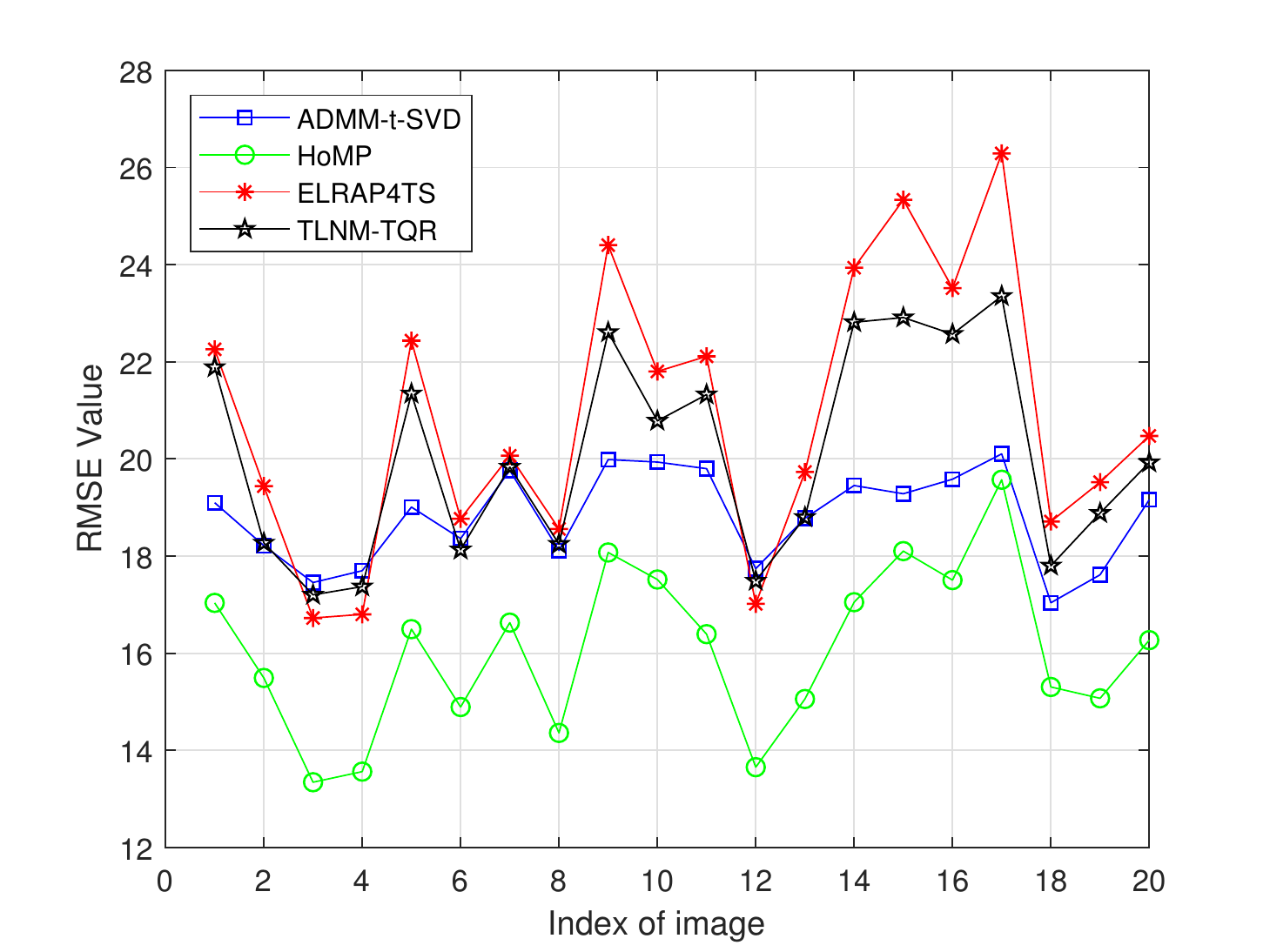}  
			\includegraphics[scale=0.45]{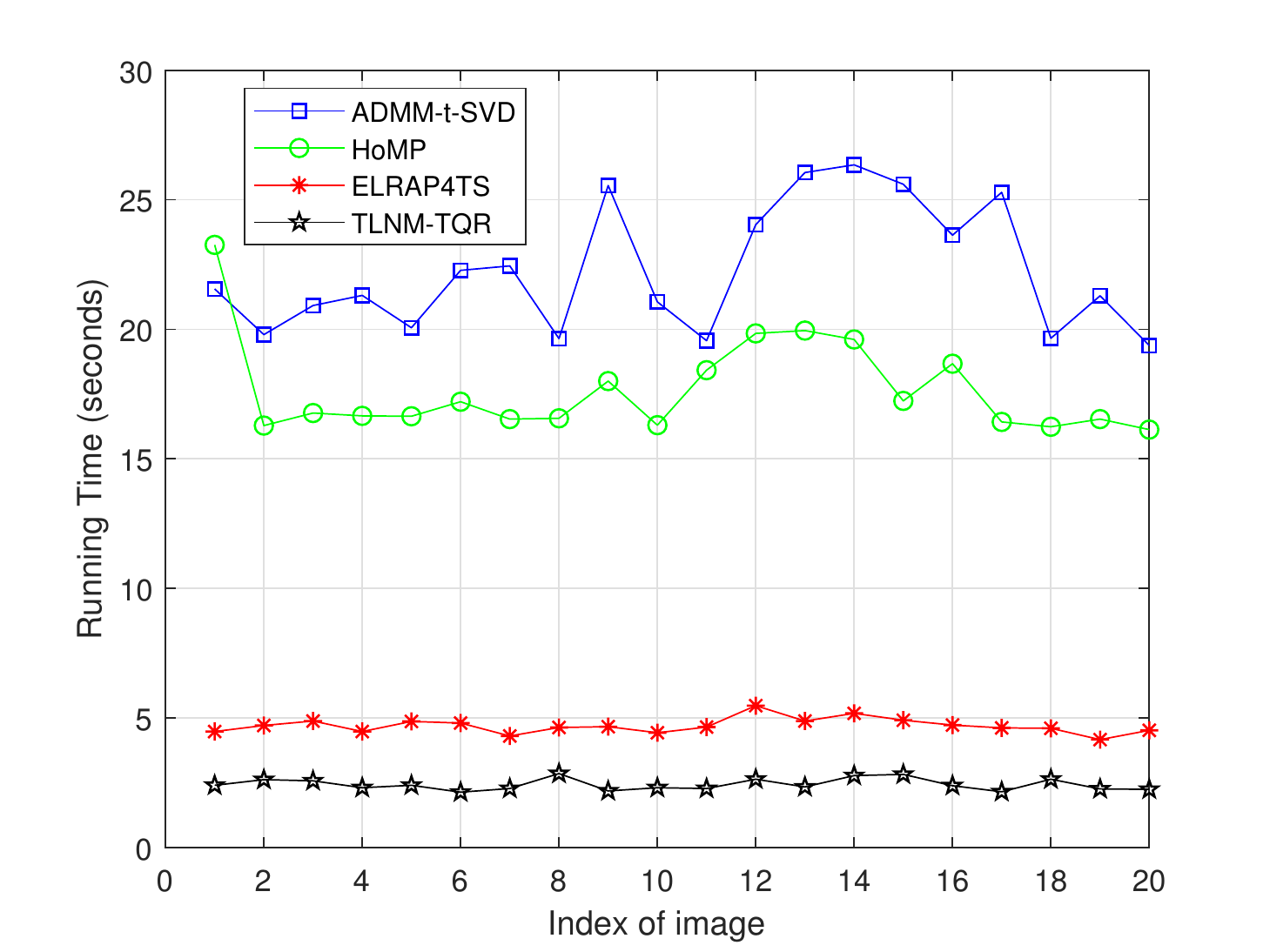}  
	
	\caption{RMSE Value $\left(\text{left column}\right)$ and Running Time $\left(\text{right column}\right)$ of recovery results of four algorithms.} 
	\label{image comparative result of TLNM-TQR}
\end{figure}

To further test the performance of our algorithm, here the Matlab function, imnoise, is used to generate blurring noise to the 20 color images and add Gaussian noise with a mean zero and a standard deviation $\sigma =5e-3$ to make images more difficult to recover. Besides, we maintain $50\%$ of given pixels as known entries in $\Omega$. The results of tests and comparison are shown in Figure \ref{image comparative result of TLNM-TQR}, and some examples of recovred images are shown in Figure \ref{Recovery image show}.

\begin{figure}[H]
	\centering
	\noindent\makebox[\textwidth][c] {
	\includegraphics[scale=0.6]{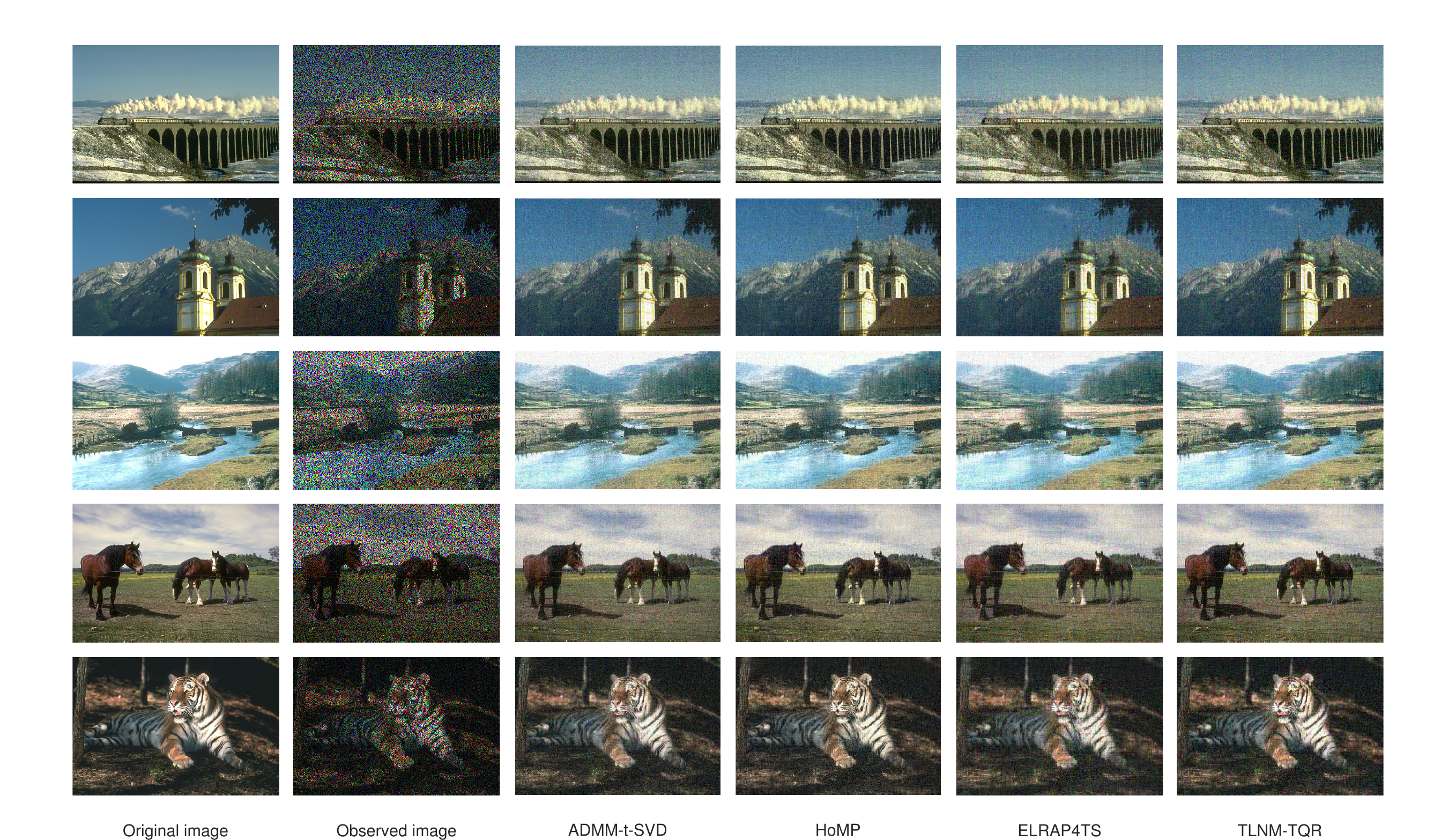} }
	\caption{Recovery performance comparison on 5 example images.}
	\label{Recovery image show}
\end{figure}

From these results and some experimental datas in Table \ref{Comparison results of running time and the RMSE on the 5 images}, TLNM-TQR is still the fastest and the effect of images recovery is satisfactory.

\begin{table}[H]
	\caption{Comparison results of running time and the RMSE on the 5 color images in Figure \ref{Recovery image show}}
	\label{Comparison results of running time and the RMSE on the 5 images}
	\resizebox{\textwidth}{!}
	{
		\begin{tabular}{ccccccccc}
			\toprule
			\multirow{2}{*}{Index} &    \multicolumn{4}{c}{Running Time (seconds)}&\multicolumn{4}{c}{RMSE Value}\\  
			\cmidrule(lr){2-5} \cmidrule(lr){6-9}
			&ADMM-t-SVD&HoMP&ELRAP4TS&TLNM-TQR&ADMM-t-SVD&HoMP&ELRAP4TS&TLNM-TQR \\
			\midrule
			1&22.8229&24.1802&4.7112&\textbf{2.4336}&19.7492&\textbf{17.1329}     & 23.2639      & 22.1667 \\
			2&22.3081&16.6141&5.0232&\textbf{2.3868}&18.1883& \textbf{15.5762}   & 19.6122     & 19.3172 \\
			3&22.1365&16.8169&5.0232&\textbf{2.6052}&20.2465&  \textbf{17.6365}     & 23.2628     & 22.2121 \\
			4&22.4329&17.4253&4.7112&\textbf{2.2308}&19.1752&\textbf{16.1871}&20.9125&20.0297 \\
			5&22.4953&16.8013&5.0388&\textbf{2.3556}&18.7036&\textbf{15.9482}&22.1168&20.8750
			\\
			\bottomrule
		\end{tabular}}
\end{table}

\section{Conclusion}

To explore a fast and precise tensor completion method and inspired by the t-product, we broke though the traditional t-SVD decomposition and proposed a new framework of tensor tri-factorization named CTSVD-QR, which can compute the largest $r$ singular values (tubes) and their associated singular vectors (of tubes) iteratively. Besides, we also abandon the previous norm and propose the tensor $L_{2,1}$-norm to simplify the optimization problem. Then, based on CTSVD-QR and tensor $L_{2,1}$-norm, TLNM-TQR is proposed for tensor completion. For CTSVD-QR and TLNM-TQR, we use synthetic and real-world datas to test its performance. Furthermore, here some state-of-the-art algorithms such as ADMM-t-SVD, HoMP and ELRAP4TS are added to the performance test of TLNM-TQR for comparison. Finally, experimental results show that our method is feasible and convergent. Moreover, the results of video and color images recovery show that TLNM-TQR is much faster than the other algorithms. The experimental results in video recovery still show that TLNM-TQR is almost as precise as the ADMM-t-SVD.


\end{document}